\documentclass[9pt]{revtex4}
\usepackage{hyperref}
\usepackage{datetime}
\usepackage{amsmath,amsfonts,amsthm,amssymb}
\usepackage{mdframed}
\usepackage{cases}
\usepackage{color}
\usepackage{graphicx,subcaption,float,wrapfig}
\usepackage{overpic}
\usepackage[usenames, dvipsnames]{xcolor}
\hypersetup{colorlinks=true,linkcolor=Blue,citecolor=Blue}
\numberwithin{equation}{section}
\begin{document}
\title{Elastic anisotropy of nematic liquid crystals in the two-dimensional Landau-de Gennes model}
\author{Yucen Han$^{1}$}
\author{Joseph Harris$^{1}$}
\author{Lei Zhang$^{2}$}
\author{Apala Majumdar$^{1}$}
\affiliation{$^1$ Department of Mathematics and Statistics, University of Strathclyde, G1 1XQ, United Kindom.\\
$^2$Beijing International Center for Mathematical Research, Center for Quantitative Biology, Peking University, Beijing 100871, China.}


\newtheorem{theorem}{Theorem}[section]
\newtheorem{proposition}{Proposition}[section]
\newtheorem{corollary}{Corollary}[section] 
\newtheorem{definition}{Definition}
\newtheorem{lemma}[proposition]{Lemma}

\theoremstyle{remark}
\newtheorem{step}{Step}
\newtheorem{remark}{Remark}

\theoremstyle{remark}
\newtheorem*{rem}{Remark}
\renewcommand{\thefigure}{\arabic{section}.\arabic{figure}}
\renewcommand{\thetable}{\arabic{section}.\arabic{table}}
\renewcommand{\theequation}{\arabic{section}.\arabic{equation}}

\setcounter{section}{0}

\newcommand{\R}{{\mathbb R}}
\newcommand{\D}{\mathrm{D}} 
\renewcommand\d{\mathrm{d}}
\newcommand\dd{\mathrm{d}}
\newcommand\nvec{\mathbf{n}}
\newcommand\e{\mathbf{e}}
\newcommand\m{\mathbf{m}}
\newcommand\p{\mathbf{p}}
\newcommand\vvec{\mathbf{v}}
\newcommand\rvec{\mathbf{r}}
\newcommand\w{\mathbf{w}}
\newcommand\x{\mathbf{x}}
\newcommand\xhat{\hat{\mathbf{x}}}
\newcommand\yhat{\hat{\mathbf{y}}}
\newcommand\zhat{\hat{\mathbf{z}}}
\newcommand\Qvec{\mathbf{Q}}
\newcommand\Gvec{\mathbf{H}}
\newcommand\Pvec{\mathbf{P}}
\newcommand\Ivec{\mathbf{I}}
\newcommand\E{\mathbf{E}}
\newcommand\pp{\partial}
\newcommand\tr{\mathrm{tr}}
\newcommand\Div{\mathrm{div}}
\newcommand\grad{\mathrm{grad}}

\newcommand{\abs}[1]{\left|#1\right|}

\newcommand\BBB{\color{blue}}

\begin{abstract}
    We study the effects of elastic anisotropy on the Landau-de Gennes critical points for nematic liquid crystals, in a square domain. The elastic anisotropy is captured by a parameter, $L_2$, and the critical points are described by three degrees of freedom. We analytically construct a symmetric critical point for all admissible values of $L_2$, which is necessarily globally stable for small domains i.e., when the square edge length, $\lambda$, is small enough. We perform asymptotic analyses and numerical studies to discover at least $5$ classes of these symmetric critical points - the $WORS$, $Ring^{\pm}$, $Constant$ and $pWORS$ solutions, of which the $WORS$, $Ring^+$ and $Constant$ solutions can be stable. Furthermore, we demonstrate that the novel $Constant$ solution is energetically preferable for large $\lambda$ and large $L_2$, and prove associated stability results that corroborate the stabilising effects of $L_2$ for reduced Landau-de Gennes critical points. We complement our analysis with numerically computed bifurcation diagrams for different values of $L_2$, which illustrate the interplay of elastic anisotropy and geometry for nematic solution landscapes, at low temperatures.
\end{abstract}

\maketitle
\section{Introduction}
\label{sec:intro}
Nematic liquid crystals (NLCs) are quintessential examples of partially ordered materials that combine fluidity with the directionality of solids \cite{prost1995physics}. The nematic molecules are typically asymmetric in shape e.g., rod- or disc-shaped, and these molecules tend to align along certain locally preferred averaged directions, referred to as \emph{nematic directors} in the literature. Consequently, NLCs have a degree of long-range orientational order and direction-dependent physical, optical and rheological properties. It is precisely this anisotropy that makes them the working material of choice for a range of electro-optic devices such as the multi-billion dollar liquid crystal display industry \cite{lagerwallreview,wang2021modeling}.

There has been substantial recent interest in multistable NLC systems i.e., NLCs, confined to two-dimensional (2D), or three-dimensional (3D) geometries that can support multiple stable states without any external electric fields \cite{robinson2017molecular, kusumaatmaja2015free, luo2012multistability, canevari2017order, canevariharrismajumdarwang,han2019transition,yin2020construction}. Multistable NLC systems offer new prospects for device technologies, materials technologies, self-assembly processes and hydrodynamics. This paper is motivated by a bistable system reported in \cite{tsakonas2007multistable}. Here, the authors experimentally and numerically study NLCs inside periodic arrays of 3D wells, with a square cross-section, such that the well height is typically much smaller than the square cross-sectional length. Furthermore, the authors speculate that the structural characteristics only vary in the plane of the square cross-section, and are translationally invariant along the well-height, effectively reducing this to a 2D problem. Hence, the authors restrict attention to the bottom square cross-section of the well geometry, where the square edge length is denoted by $\lambda$ and is typically on the micron scale. The choice of boundary conditions is crucial for any confined NLC system and, in \cite{tsakonas2007multistable}, the authors impose tangent boundary conditions (TBCs) on the well surfaces i.e., the nematic directors, in the plane of the well surfaces, are constrained to be tangent to those well surfaces. However, this necessarily means that the nematic director has to be tangent to the square edges, creating defects at the vertices where the director is not defined. The authors observe two classes of stable NLC states in this geometry: the diagonal $D$ states, for which the nematic director aligns along one of the square diagonals and; the rotated $R$ states, for which the director rotates by $\pi$ radians between a pair of opposite square edges.

In \cite{kralj2014order}, \cite{luo2012multistability}, the authors model this square system within the celebrated continuum Landau-de Gennes (LdG) theory for NLCs. The LdG theory describes the state of nematic anisotropy by a macroscopic order parameter - the $\Qvec$-tensor order parameter \cite{prost1995physics}. From an experimental perspective, the $\Qvec$-tensor is measured in terms of NLC responses to external electric or magnetic fields, which are necessarily anisotropic in nature. Mathematically, the $\Qvec$-tensor order parameter is a symmetric, traceless, $3\times 3$ matrix whose eigenvectors represent the special directions of preferred molecular alignment, and the corresponding eigenvalues measure the degree of orientational order about these directions; more details are given in the next section. The nematic director can be identified with the eigenvector that has the largest positive eigenvalue. For a square domain with TBCs on the square edges, it suffices to work in a reduced LdG framework where the $\Qvec$-tensor only has three degrees of freedom, $q_1, q_2, q_3$. The degree of nematic order in the plane is captured by $q_1$ and $q_2$, whereas $q_3$ measures the out-of-plane order, such that positive (negative) $q_3$ implies that the nematic director lies out of the plane (in the plane) of the square, respectively. The TBCs naturally constrain $q_3$ to be negative on the square edges, but $q_3$ could be positive in the interior, away from the square edges, for energetic reasons. The LdG theory is a variational theory i.e., experimentally observable states can be modelled by local or global minimizers of an appropriately defined LdG free energy. In the simplest setting, the LdG energy has two contributions - a bulk energy that only depends on the eigenvalues of the $\Qvec$-tensor, and an elastic energy that penalises spatial inhomogeneities of the $\Qvec$-tensor. In these papers, the authors work with low temperatures that favour an ordered nematic state for a spatially homogeneous system i.e., the bulk energy attains its minimum at an ordered nematic state for low temperatures, and attains its minimum at a disordered isotropic phase for high temperatures. The elastic energy is typically a quadratic and convex function of $\nabla \Qvec$ and, in \cite{kralj2014order}, \cite{luo2012multistability}, the authors work with an isotropic elastic energy - the Dirichlet elastic energy. In a reduced LdG setting, the authors recover the stable $D$ and $R$ states for large $\lambda$ and, in \cite{kralj2014order}, they discover a novel stable Well Order Reconstruction Solution ($WORS$) for small $\lambda$. The $WORS$ is special because it exhibits a pair of mutually orthogonal defect lines, with no planar nematic order, along the square diagonals as will be described in Section~\ref{sec:qualitative}. In \cite{han2020pol}, the authors generalise this work to arbitrary 2D regular polygons and, in \cite{canevariharrismajumdarwang}, the authors study 3D wells, with an emphasis on novel mixed solutions which interpolate between two distinct $D$ solutions on the top and bottom well surfaces.

In this paper, we study the same problem of NLCs on a square domain with TBCs on the square edges, with an anisotropic elastic energy as opposed to the isotropic energy studied in \cite{kralj2014order} and \cite{luo2012multistability}. Notably, we take the elastic energy density to be $w(\nabla \Qvec) = | \nabla \Qvec |^2 + L_2 \left(\textrm{div} \Qvec \right)^2$, where $-1 < L_2$ is the anisotropy parameter. Physically speaking, positive $L_2$ implies that splay and twist deformations of the nematic director are energetically expensive compared to out-of-plane twist deformations i.e., we would expect the physically observable states to have positive $q_3$ in the square interior as $L_2$ increases. Therefore, we would expect to see competing effects of the TBCs on the square edges, which prefer in-plane  director orientation, and the preferred out-of-plane director orientation in the square interior, for larger values of $L_2$. We construct a critical point of the LdG energy, for any $L_2> -1$, for which $q_1 = 0$ on the square diagonals and $q_2=0$ on the coordinate axes such that the axes are parallel to the square edges. This symmetric critical point is globally stable for $\lambda$ small enough. The $WORS$ is a special case of this symmetric critical point with $q_2 \equiv 0$ on the square domain, for $L_2 = 0$. For $L_2 \neq 0$, this class of symmetric critical points cannot have $q_2$ identically zero on the domain, and this destroys the perfect cross symmetry of the $WORS$. We perform an asymptotic analysis for small $\lambda$ and small $L_2$, about the $WORS$, which is the globally stable branch in this regime for $L_2 = 0$. The anisotropy has a first order effect on $q_3$ i.e., $q_3$ is perturbed linearly by $L_2$, whereas $q_1$ and $q_2$ exhibit quadratic perturbations. We show that $q_3$ increases at the square centre for positive $L_2$, relative to its value for $L_2 = 0$, corroborating the trend of increasing $q_3$ with increasing $L_2$. The globally stable symmetric critical point for small $\lambda$ and small $L_2$, labelled as the $Ring^+$ solution, effectively smoothens out the $WORS$ and exhibits a stable central $+1$-degree point defect. We perform formal calculations to show that as $L_2 \to \infty$, energy minimizers (and consequently the symmetric critical point described above for small $\lambda$) approach the $Constant$ state with $(q_1, q_2, q_3) = (0, 0, s_+/3)$ away from the square edges. The special choice of $q_3 = s_+/3$ stems from energy minimality and the choice of TBCs on the square edges. Thus, there are three different classes of the symmetric critical point discussed above: the $WORS$, which only exists for $L_2= 0$; the $Ring^+$ solution, which only ceases to exist for $L_2$ large enough, and can be stable for moderate values of $\lambda$ and non-zero $L_2$ and; the $Constant$ solution, which exists for $L_2$ large enough and is always stable according to our heuristics and numerical calculations. Additionally, we also find two unstable classes of this symmetric critical point, both of which exist for moderate values of $\lambda$ and $L_2$. These are the $Ring^-$ solution which exhibits a central $-1$-degree point defect, and the novel $pWORS$ which exhibits an oscillating sequence of nematic point defects along the square diagonals. We provide asymptotic approximations for the novel $pWORS$ solution branch.

Whilst most of our work is restricted to the small $\lambda$-limit, we also provide rigorous results for energy minimizers in the $\lambda \to \infty$ limit. The competitors in the large $\lambda$-limit are the familiar $D$ and $R$ states, and the $Constant$ solution. Using Gamma-convergence arguments, we show that the $Constant$ solution has lower energy than the $D$ and $R$ states, for large enough $L_2$. We complement our analysis with numerical computations of bifurcation diagrams for five different values of $L_2$. As $L_2$ increases, we observe that the unique minimizer for small $\lambda$ changes from the $WORS$ ($L_2 = 0$), to the $Ring^+$ and then to the $Constant$ solution. As $L_2$ increases further, the $Ring^+$ and $Constant$ solutions retain stability in the reduced framework, over an increasing range of $\lambda$. We further prove this by performing an analysis of the corresponding second variation of the reduced LdG energy. It is interesting to note that whilst the $WORS$ and $Ring^+$ solution branches are connected to the $D$ and $R$ states, the $Constant$ solution branch appears to be disconnected from the $D$ and $R$ solution branches. Our notable findings concern the response of the NLC solution landscape for this model problem, to the elastic anisotropy $L_2$. We report (i) novel stable states ($Ring^+$ and $Constant$) for small $\lambda$, and (ii) enhanced multistability in the large $\lambda$-limit due to the competing $Constant$ and $Ring^+$ states, for large $L_2$. As $L_2$ increases, we expect that there are further, not necessarily energy-minimizing, LdG critical points with positive $q_3$, or out-of-plane nematic directors. Furthermore, $L_2$ has a stabilising effect with respect to certain classes of perturbations: planar perturbations and out-of-plane perturbations, and so we expect enhanced multistability as $L_2$ increases, for all values of $\lambda$.

A lot of open questions remain with regards to the interplay between $L_2$, $\lambda$ and temperature on NLC solution landscapes, but our work is an informative forward step in this direction. Our paper is organised as follows. We provide all the mathematical preliminaries in Section~\ref{sec:prelim}. We construct the symmetric critical points described above and prove their global stability for small $\lambda$ in Section~\ref{sec:qualitative}. In Section~\ref{sec:asymptotics}, we perform separate asymptotic studies in the small $\lambda$ and small $L_2$ limit; large $L_2$ limit; large $\lambda$-limit. In Section~\ref{sec:bifurcations}, we present bifurcation diagrams for the solution landscapes with five different values of $L_2$, accompanied by some rigorous stability results. We conclude with some perspectives in Section~\ref{sec:conclusions}. 

\section{Preliminaries}
\label{sec:prelim}
In this section, we review the Landau-de Gennes (LdG) continuum theory of nematic liquid crystals. Within this  framework, the nematic state is described by a macroscopic LdG order parameter - the $\Qvec$-tensor order parameter. The $\Qvec$-tensor is a symmetric, traceless, $3\times3$ matrix, which is a macroscopic measure of the nematic anisotropy. The eigenvectors of $\Qvec$ represent the averaged directions of preferred molecular alignment and the corresponding eigenvalues measure the degree of order about these eigen-directions. The $\Qvec$-tensor is said to be: (i) isotropic if $\Qvec=0$; (ii) uniaxial if $\Qvec$ has a pair of degenerate non-zero eigenvalues; and (iii) biaxial if $\Qvec$ has three distinct eigenvalues. A uniaxial $\Qvec$-tensor can be written as $\Qvec_u=s\left(\nvec\otimes\nvec-\Ivec/3\right)$,
where $\Ivec$ is the $3\times3$ identity matrix, $\nvec \in S^2$ is the distinguished eigenvector with the non-degenerate eigenvalue, and  $s\in\mathbb{R}$ is a scalar order parameter that measures the degree of orientational order about $\nvec$. The unit vector, $\nvec$, is referred to as the ``director", and physically denotes the single distinguished direction of uniaxial nematic alignment \cite{Virga94}, \cite{prost1995physics}. 
The LdG theory is a variational theory, and hence, has an associated free energy, and the basic modelling hypothesis is that the physically observable configurations correspond to global or local energy minimizers subject to imposed boundary conditions. We work with two-dimensional (2D) domains, $\Omega \subset \mathbb{R}^2$, in the context of modelling \emph{thin} three-dimensional (3D) systems.
In the absence of a surface anchoring energy and external electric/magnetic fields, the LdG free energy is given by
\begin{gather}
    \mathcal{F}[\Qvec]:=\int_{\Omega}f_{el}(\Qvec,\nabla\Qvec)+f_b(\Qvec)\,\mathrm{dA}, \label{ldgunres}
\end{gather}
where $f_{el}$ and $f_b$ are the elastic and thermotropic bulk energy densities, respectively. We consider a two-term elastic energy density given by
\begin{gather}
    f_{el}(\Qvec)=\frac{L}{2}\left(|\nabla\Qvec|^2+L_2(\Div{\Qvec})^2\right),
\end{gather}
where $L>0$ is an elastic constant, and $L_2\in(-1,\infty)$ is the ``elastic anisotropy'' parameter. The elastic energy density penalises spatial inhomogeneities, typically quadratic in $\nabla \Qvec$.  In terms of notation, we use $|\nabla\Qvec|^2:=\frac{\partial Q_{ij}}{\partial x_k}\frac{\partial Q_{ij}}{\partial x_k}$ and $(\Div{\Qvec})^2:=\frac{\partial Q_{ij}}{\partial x_j}\frac{\partial Q_{ik}}{\partial x_k}$, $i,j,k =1,2,3,$ where the Einstein summation convention is assumed throughout this manuscript. Since in this work we assume a 2D confining geometry $\Omega$, we have that $\frac{\partial Q_{ij}}{\partial x_3}=0$ for all $1\leq i,j\leq 3$. 
We work with the simplest form of $f_b$, that allows for a first-order isotropic-nematic transition as a function of the temperature:
\begin{gather}
    f_b(\Qvec):=\frac{A}{2}\tr\Qvec^2-\frac{B}{3}\tr\Qvec^3+\frac{C}{4}(\tr\Qvec^2)^2. \label{bulk}
\end{gather}
Here, $\tr\Qvec^2=Q_{ij}Q_{ij}$, and $\tr\Qvec^3=Q_{ij}Q_{jk}Q_{ki}$, for $i,j,k=1,2,3$. We take $A=\alpha(T-T^*)$ to be the rescaled temperature and $\alpha, B, C>0$ are material-dependent constants. In this regime, $T$ is the absolute temperature in the system, and $T^*$ is the characteristic nematic supercooling temperature. The rescaled temperature, $A$, has three physically relevant values: (i) $A=0$, below which the isotropic state $\Qvec=0$ loses stability; (ii) the nematic super-heating temperature $A=B^2/24C$, above which the isotropic state is the unique critical point of $f_b$; and (iii) the nematic-isotropic phase transition temperature $A=B^2/27C$, at which $f_b$ is minimized by the isotropic phase and a continuum of uniaxial states. 
We work with low temperatures, $A<0$, for which $f_b$ is minimized on the set of uniaxial $\Qvec$-tensors defined by 
 $\mathcal{N}:=\{\Qvec\in S_0 : \Qvec=s_+(\nvec\otimes\nvec-\Ivec/3)\}$  where $S_0$ is the space of traceless symmetric $3\times3$ matrices and
\begin{gather}
    s_+=\frac{B+\sqrt{B^2+24|A|C}}{4C},\quad \nvec\in S^2 \,\, \text{arbitrary}.
\end{gather}
We non-dimensionalize the system using a change of variables, $\mathbf{\bar{x}}=\mathbf{x}/\lambda$, where $\lambda$ is a characteristic geometrical length-scale e.g., edge length of a 2D regular polygon. The rescaled LdG energy functional (upto a multiplicative constant) is given by:
\begin{gather}
    \mathcal{F}_{\lambda}[\Qvec]:=
    \int_{\bar{\Omega}}\left\{\frac{1}{2}|\nabla_{\mathbf{\bar{x}}}\Qvec|^2+\frac{L_2}{2}(\Div_{\mathbf{\bar{x}}}\Qvec)^2+\frac{\lambda^2}{L}f_b(\Qvec)\right\}\,\mathrm{d\bar{A}}, \label{rescE}
\end{gather}
where $\bar{\Omega}$ is the rescaled domain in $\mathbb{R}^2$, and $\mathrm{d\bar{A}}$ is the rescaled area element. We drop the `bars' but all computations should be interpreted in terms of the rescaled variables. 

Next, we define the working domain and Dirichlet boundary conditions for the purposes of this paper, although we believe that our methods can be generalised to arbitrary 2D domains, and other types of Dirichlet conditions. We focus on square domains, building on the substantial work in \cite{Walton2018}, \cite{canevari2017order}, \cite{wang2019order}. We impose Dirichlet tangent boundary conditions (TBCs) on the square edges, which require the nematic director to be tangent to the edges, necessarily creating a mismatch at the square vertices. To avoid the discontinuities at the vertices, we take $\Omega\subset\mathbb{R}^2$ to be a truncated square whose edges are parallel to the coordinate axes:
\begin{gather}
    \Omega:=\{(x,y)\in\mathbb{R}^2 : |x|<1, |y|<1, |x+y|<2-\epsilon, |x-y|<2-\epsilon\}. \label{truncsq}
\end{gather}
Provided $\epsilon \ll 1$, the truncation does not change the qualitative properties of the LdG energy minimizers away from the square vertices. The boundary, $\partial\Omega$, has four ``long'' edges parallel to the coordinate axes which we define in a clockwise fashion as $C_1,\dots,C_4$, where $C_1$ lies parallel to the $x$-axis at $y=1$. The truncation creates four additional ``short'' edges, of length $\sqrt{2}\epsilon$, parallel to the lines $y=x$ and $y=-x$, which we label as $S_1,\dots,S_4$  in a clockwise fashion, starting at the top-left corner of the domain. The domain is illustrated in Figure \ref{fig:square}.
\begin{figure}
    \centering
    \includegraphics[width=0.3\columnwidth]{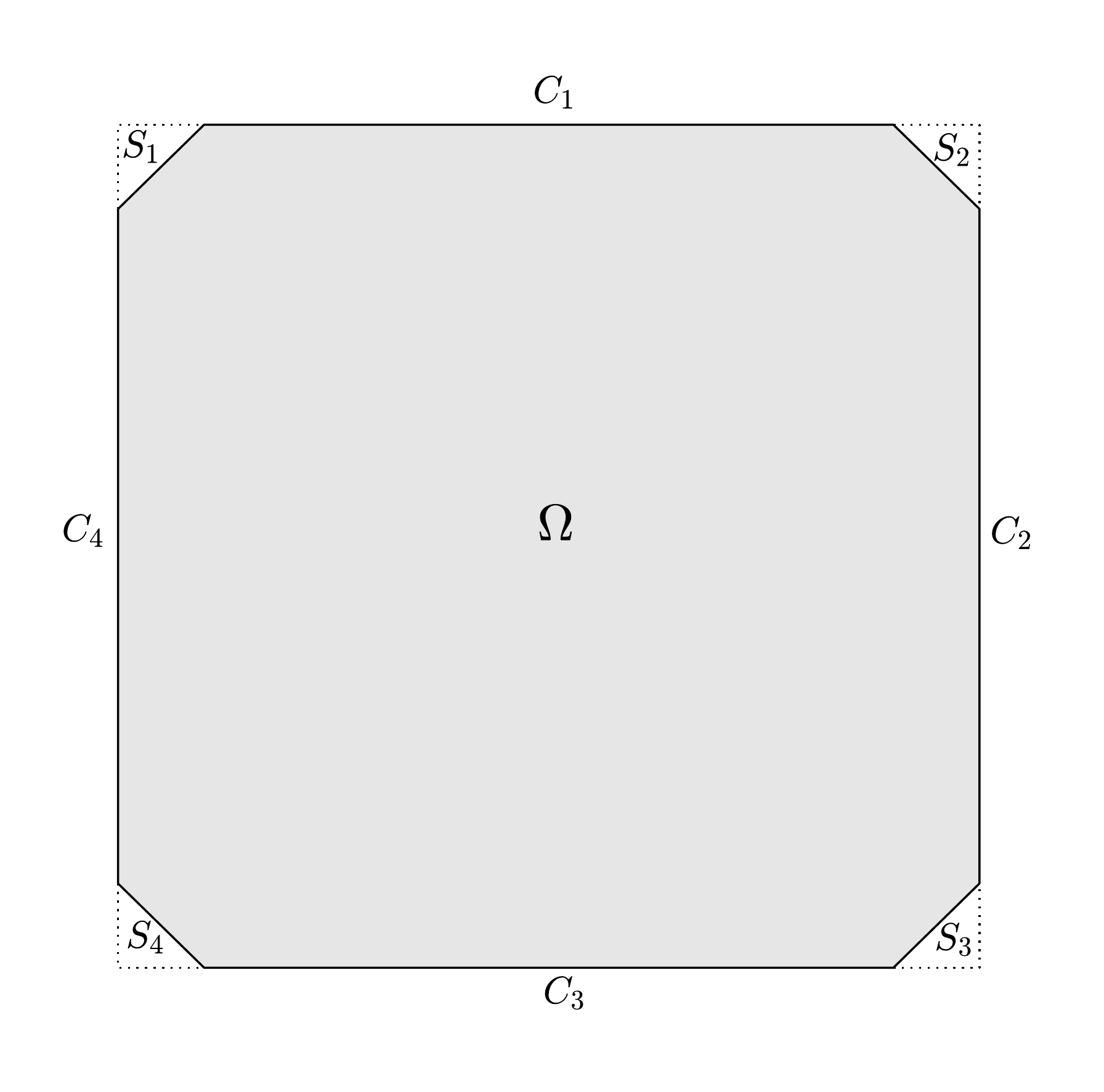}
    \caption{The truncated square domain $\Omega$.}
    \label{fig:square}
\end{figure}

We impose tangent uniaxial Dirichlet conditions on the long edges, consistent with the experimentally and numerically investigated TBCs, \cite{tsakonas2007multistable}, \cite{luo2012multistability} and \cite{kralj2014order}. In particular, we fix $\nvec = \left(\pm 1, 0 \right)$ on the edges, $C_1$ and $C_3$, and $\nvec = \left(0, \pm 1 \right)$ on $C_2$ and $C_4$. From a physical standpoint, this constitutes \textit{strong} (infinite) anchoring on the long edges. One could also model weak (finite) anchoring condition with an additional surface energy in the LdG free energy \cite{newtonmottram}, but that would make the analysis more complicated for the time being. We set
\begin{gather}
    \Qvec=\Qvec_b\qquad\text{on}\qquad\partial\Omega, \label{BCQ}
\end{gather}
where \begin{align}\label{longedgebc}
    \Qvec_b(x,y):=
    \begin{cases}
    s_+\left(\xhat\otimes\xhat-\Ivec/3\right),\qquad& (x,y)\in C_1\cup C_3, \\
    s_+\left(\yhat\otimes\yhat-\Ivec/3\right),\qquad& (x,y)\in C_2\cup C_4,
    \end{cases}
\end{align}
where $\xhat$ and $\yhat$ are unit vectors in the $x$- and $y$-direction, respectively. In particular, $\Qvec_b \in \mathcal{N}$ on $C_1, \ldots, C_4$. On the short edges, $S_1,\dots,S_4$, we effectively prescribe a continuous interpolation between the boundary conditions on the associated long edges \eqref{longedgebc} given by:
\begin{gather} \label{shortedgebc}
    \Qvec_b(x,y):=
    \begin{cases}
    g(x+y)(\xhat\otimes\xhat-\yhat\otimes\yhat)-\frac{s_+}{2}(\zhat\otimes\zhat-\Ivec/3),\quad&(x,y)\in S_1, \\
    g(y-x)(\xhat\otimes\xhat-\yhat\otimes\yhat)-\frac{s_+}{2}(\zhat\otimes\zhat-\Ivec/3),\quad&(x,y)\in S_2, \\
    g(-x-y)(\xhat\otimes\xhat-\yhat\otimes\yhat)-\frac{s_+}{2}(\zhat\otimes\zhat-\Ivec/3),\quad&(x,y)\in S_3, \\
    g(x-y)(\xhat\otimes\xhat-\yhat\otimes\yhat)-\frac{s_+}{2}(\zhat\otimes\zhat-\Ivec/3),\quad&(x,y)\in S_4,
    \end{cases}
\end{gather}
where $\zhat$ is a unit vector in the $z$-direction, and $g:[-\epsilon,\epsilon]\to[-s_+/2,s_+/2]$ is a smoothing function defined as 
\begin{gather}
    g(l)=\frac{s_+}{2\epsilon}l, \qquad -\epsilon\leq l\leq\epsilon.
\end{gather}
Although the boundary conditions \eqref{shortedgebc} do not minimize $f_b$ on $S_1,\dots,S_4$, and do not respect TBCs, they are short by construction and are chosen purely for mathematical convenience. Given the Dirichlet boundary conditions (\ref{longedgebc}) and (\ref{shortedgebc}), we define our admissible space to be
\begin{gather}
    \mathcal{A}:=\{\Qvec\in W^{1,2}(\Omega; S_0) : \Qvec=\Qvec_b\,\, \,\,\text{on}\,\,\, \partial\Omega\}. \label{admissible}
\end{gather} 
The energy minimizers, or indeed any critical point of the LdG energy (\ref{rescE}), are solutions of the associated Euler-Lagrange equations: 
\begin{gather}
    \Delta Q_{ij}+\frac{L_2}{2}\left(Q_{ik,kj}+Q_{jk,ki}-\frac{2}{3}\delta_{ij}Q_{kl,kl}\right)=\frac{\lambda^2}{L}\left\{AQ_{ij}-B\left(Q_{ik}Q_{kj}-\frac{1}{3}\delta_{ij}\tr{\Qvec^2}\right)+CQ_{ij}\tr{\Qvec^2}\right\},\label{EL}
\end{gather}
which comprise a system of five nonlinear coupled partial differential equations. The terms $\frac{2}{3}Q_{kl,kl}$ and $\frac{1}{3}\tr{\Qvec^2}$ are Lagrange multipliers associated with the tracelessness constraint.

Finally, we comment on the physical relevance of the 2D domain, $\Omega \subset \mathbb{R}^2$. Consider a 3D well,
$$
\mathcal{B} = \left\{(x,y,z) \in \mathbb{R}^3; (x,y) \in \Omega; z \in (0, h) \right\},
$$
where $h \ll \lambda$, and $\lambda$ is a characteristic length scale associated with $\Omega$. In this limit, one can assume (at least for modelling purposes) that  physically relevant $\Qvec$-tensors are independent of the $z$-coordinate i.e., the profiles are invariant across the height of the well, and that $\zhat$ is a fixed eigenvector (see \cite{golovaty2017dimension} and \cite{bauman2012analysis} for some rigorous analysis and justification). This implies that we can restrict ourselves to $\Qvec$-tensors with three degrees of freedom:
\begin{gather}
    \Qvec=q_1(x,y)(\xhat\otimes\xhat-\yhat\otimes\yhat)+q_2(x,y)(\xhat\otimes\yhat+\yhat\otimes\xhat) \nonumber\\
    +q_3(x,y)(2\zhat\otimes\zhat-\xhat\otimes\xhat-\yhat\otimes\yhat)\label{q123},
\end{gather}
subject to the boundary conditions
\begin{gather}
    q_1(x,y)=q_b(x,y)=
    \begin{cases}
    s_+/2,\quad&\text{on}\quad C_1\cup C_3, \\
    -s_+/2,\quad&\text{on}\quad C_2\cup C_4, \\
    g(x+y),\quad&\text{on}\quad S_1,\\
    g(y-x),\quad&\text{on}\quad S_2,\\
    g(-x-y),\quad&\text{on}\quad S_3,\\
    g(x-y),\quad&\text{on}\quad S_4, 
    \end{cases} \label{bcq1}
\end{gather}
and 
\begin{gather}
    q_2=0,\quad\text{ and}\quad q_3(x,y)=-s_+/6\quad\text{ on}\quad \partial\Omega. \label{bcq2q3}
\end{gather}
The conditions (\ref{bcq1}) and (\ref{bcq2q3}) are equivalent to Dirichlet conditions in (\ref{BCQ}).

\section{Qualitative Properties of Equilibrium Configurations}
\label{sec:qualitative}
 In \cite{kralj2014order}, the authors numerically compute critical points of (\ref{rescE}) with $L_2=0$, satisfying the Dirichlet boundary conditions (\ref{BCQ}), on the square cross-section $\Omega$ with edge length $\lambda$. For $\lambda$ small enough, the authors report a new Well Order Reconstruction Solution ($WORS$). The $WORS$ has a constant set of eigenvectors, $\xhat, \yhat$, and $\mathbf{\hat{z}}$, which are the coordinate unit vectors. The $WORS$ is further distinguished by a uniaxial cross, with negative scalar order parameter, along the square diagonals. Physically, this implies that  there is a planar defect cross along the square diagonals, and the nematic molecules are disordered along the square diagonals. In \cite{canevari2017order}, the authors analyse this system at a fixed temperature $A=-B^2/3C$ with $L_2 = 0$, and show that the $WORS$ is a classical solution of the associated Euler-Lagrange equations (\ref{EL}) of the form:
\begin{align}
\mathbf{Q}_{WORS}(x,y)=q(\xhat\otimes\xhat-\yhat\otimes\yhat)-\frac{B}{6C}(2\mathbf{\hat{z}}\otimes\mathbf{\hat{z}}-\xhat\otimes\xhat-\yhat\otimes\yhat). \label{WQ}
\end{align}
There is a single degree of freedom, $q:\Omega\to\mathbb{R}$, which satisfies the Allen-Cahn equation and exhibits the following symmetry properties:
\begin{gather}
    q=0\quad\text{on}\quad\{y=x\}\cup\{y=-x\},\qquad
    (y^2-x^2)q(x,y)\geq0.
\end{gather}
Notably, $q_2=0$ everywhere for the $WORS$ (refer to (\ref{q123})), which is equivalent to having a set of constant eigenvectors in the plane of $\Omega$.
They prove that the $WORS$ is globally stable for $\lambda$ small enough, and unstable for $\lambda$ large enough, demonstrating a pitchfork bifurcation in a scalar setting. Their analysis is restricted to the specific temperature and, in \cite{canevariharrismajumdarwang}, the authors extend the analysis to all $A<0$, with $L_2=0$. In this section, we analyse the equilibrium configurations with $L_2 \neq 0$, including their symmetry properties in the small $\lambda$ limit. Notably, we show that the cross structure of the $WORS$ does not survive with $L_2 \neq 0$, in the following propositions.

\begin{proposition}\label{prop1}
There exists at least one solution to the Euler-Lagrange equations (\ref{EL}) of the form (\ref{q123}) in $\mathcal{A}$, given the Dirichlet boundary conditions (\ref{longedgebc}) and (\ref{shortedgebc}), provided the functions $q_1, q_2, q_3$ satisfy the following systems of PDEs:
\begin{align}
    \left(1+\frac{L_2}{2}\right)\Delta q_1+\frac{L_2}{2}(q_{3,yy}-q_{3,xx})=&\frac{\lambda^2}{L}q_1(A+2Bq_3+2C(q_1^2+q_2^2+3q_3^2)), \label{q1eq} \\
    \left(1+\frac{L_2}{2}\right)\Delta q_2-L_2q_{3,xy}=&\frac{\lambda^2}{L}q_2(A+2Bq_3+2C(q_1^2+q_2^2+3q_3^2)), \label{q2eq}\\
    \left(1+\frac{L_2}{6}\right)\Delta q_3+\frac{L_2}{6}(q_{1,yy}-q_{1,xx})-\frac{L_2}{3}q_{2,xy}=&\frac{\lambda^2}{L}q_3(A-Bq_3+2C(q_1^2+q_2^2+3q_3^2))+\frac{\lambda^2B}{3L}(q_1^2+q_2^2),\label{q3eq}
\end{align}
and the boundary conditions (\ref{bcq1}) and (\ref{bcq2q3}).
\end{proposition}
\begin{proof}
Our proof is analogous to Theorem 2.2 in \cite{bauman2012analysis}. Substituting the $\Qvec$-tensor ansatz \eqref{q123} into the general form of the LdG energy \eqref{rescE}, let
\begin{align}
    J[q_1,q_2,q_3]:=&\int_{\Omega}f_{el}(q_1,q_2,q_3)+\frac{\lambda^2}{L}f_b(q_1,q_2,q_3)\,\mathrm{dA},\label{funcq123}
\end{align}
where
\begin{align}
    f_{el}(q_1,q_2,q_3):=&\left(1+\frac{L_2}{2}\right)|\nabla q_1|^2+\left(1+\frac{L_2}{2}\right)|\nabla q_2|^2+\left(3+\frac{L_2}{2}\right)|\nabla q_3|^2 \nonumber\\
    &+L_2(q_{1,y}q_{3,y}-q_{1,x}q_{3,x}-q_{2,y}q_{3,x}-q_{2,x}q_{3,y})+|L_2|(q_{2,y}q_{1,x}-q_{1,y}q_{2,x}), 
    \end{align}
    and
    \begin{align}
    f_b(q_1,q_2,q_3):=&A(q_1^2+q_2^2+3q_3^2)+C(q_1^2+q_2^2+3q_3^2)^2+2Bq_3(q_1^2+q_2^2-q_3^2),
\end{align}
are the elastic and thermotropic bulk energy densities, respectively. We prove the existence of minimizers of $J$ in the admissible class
\begin{gather} \label{Ao}
    \mathcal{A}_0:=\{(q_1,q_2,q_3)\in W^{1,2}(\Omega;\mathbb{R}^3) : q_1=q_b,\, q_2=0,\, q_3=-s_+/6\,\,\text{on}\,\,\pp\Omega\},
\end{gather} which will also be solutions of (\ref{EL}) in the admissible space, $\mathcal{A}$.
Since the boundary conditions (\ref{bcq1}) and (\ref{bcq2q3}) are piece-wise of class $C^1$, we have that the admissible space $\mathcal{A}_0$ is non-empty. The next step, is to check that $J$ is coercive in $\mathcal{A}_0$. 
The elastic energy density can be rewritten as a function of $(q_1,q_2,q_3)\in W^{1,2}(\Omega;\mathbb{R}^3)$ in the following two ways: 
\begin{gather}
    f_{el}=|\nabla q_1|^2+|\nabla q_2|^2+3|\nabla q_3|^2+\frac{L_2}{2}((q_{1,x}+q_{2,y}-q_{3,x})^2+(q_{2,x}-q_{1,y}-q_{3,y})^2), \label{pos}
    \end{gather}
    if $L_2\in[0,\infty)$, and
    \begin{gather}
    f_{el} =(1+L_2)(|\nabla q_1|^2+|\nabla q_2|^2+3|\nabla q_3|^2)-\frac{L_2}{2}((-q_{3,x}-q_{1,x}-q_{2,y})^2+(q_{2,x}-q_{1,y}+q_{3,y})^2+4|\nabla q_3|^2),\label{neg}
\end{gather}
if $L_2\in(-1,0)$. The difference between the expressions for $f_{el}$ in \eqref{pos} and \eqref{neg}, is a null Lagrangian, and hence can be ignored under the Dirichlet boundary condition.
Since we assume that $1+L_2>0$, we see that the elastic energy density can be written as the sum of non-negative terms for any $L_2>-1$ and, more specifically,
\begin{gather}
    f_{el}(q_1,q_2,q_3)\geq\min\{1,1+L_2\}\left(|\nabla q_1|^2+|\nabla q_2|^2+3|\nabla q_3|^2\right),\qquad L_2\geq0. \label{posl2}
\end{gather}
Furthermore, the bulk energy potential $f_b$ also satisfies
\begin{gather}
    f_b(q_1,q_2,q_3)\geq f_b(\pm\frac{s_+}{2},0,-\frac{s_+}{6})=:M_1(A,B,C),
\end{gather}
for some constant, $M_1>0$, depending only on the material-dependent parameters $A,B$, and $C$.
Hence $J[q_1,q_2,q_3]$ is coercive in $\mathcal{A}_0$. Finally, we note that $J$ is weakly lower semi-continuous on $W^{1,2}(\Omega)$, which follows immediately from the fact that $f_{el}$ is quadratic and convex in $\nabla(q_1,q_2,q_3)$. 
Thus, the direct method in the calculus of variations yields the existence of a global minimizer of the functional $J$ among the finite energy triplets $(q_1,q_2,q_3)\in W^{1,2}(\Omega; \mathbb{R}^3)$, satisfying the boundary conditions (\ref{bcq1})--(\ref{bcq2q3})  \cite{Evans49}. 
One can verify that the semilinear elliptic system (\ref{q1eq})--(\ref{q3eq}) corresponds to the Euler-Lagrange equations associated with $J$, and the minimizers for $J$ are $C^{\infty}(\Omega)\cap C^2(\bar{\Omega})$ solutions of (\ref{q1eq})--(\ref{q3eq}). The corresponding $\Qvec$-tensor (\ref{q123}) is an exact solution of the LdG Euler-Lagrange equations (\ref{EL}).
\end{proof}

\begin{proposition} \label{prop:forever_critical}
    There exists a critical point $(q_1^s, q_2^s, q_3^s)$ of the energy functional (\ref{funcq123}) in the admissible space $\mathcal{A}_0$, for all $\lambda>0$, such that $q_1=0$ on the square diagonals $y=x$ and $y=-x$, and $q_2=0$ on $x=0$ and $y=0$. 
\end{proposition}
\begin{proof}
We follow the approach in \cite{canevari2017order}. We define the following $1/8^{th}$ of a square located in the positive quadrant of $\Omega$:
\begin{gather}
    \Omega_q:=\{(x,y)\in\Omega : 0 < y<x,\,\, 0< x< 1\}.
\end{gather}
The following boundary conditions on $\Omega_q$ are consistent with the boundary conditions (\ref{bcq1}) and (\ref{bcq2q3}) on the whole of $\Omega$: 
\begin{gather}\label{ANquadrantBCs}
    \begin{cases}
    q_1=q_b, \, q_2=0,\, q_3=-\frac{s_+}{6},\quad & (x,y)\in\pp\Omega_q\cap\pp\Omega; \\
    q_1=\pp_\nu q_2=\pp_\nu q_3=0,\quad & (x,y)\in\pp\Omega_q\cap\{y=x\};\\
    \pp_\nu q_1=q_2=\pp_\nu q_3=0,\quad & (x,y)\in\pp\Omega_q\cap\{y=0\},\\
    \end{cases}
\end{gather}
where $\pp_\nu$ represents the outward normal derivative. We minimize the associated LdG energy functional in $\Omega_q$, given by:
\begin{gather}
J[q_1,q_2,q_3]=\int_{\Omega_q}f_{el}(q_1,q_2,q_3)+\frac{\lambda^2}{L}f_b(q_1,q_2,q_3)\,\mathrm{dA}, 
\end{gather}
on the admissible space
\begin{gather}
\mathcal{A}_q:=\{(q_1,q_2,q_3)\in W^{1,2}(\Omega_q;\mathbb{R}^3): (\ref{ANquadrantBCs})\quad \textrm{is satisfied}\}.   
\end{gather}
As the boundary conditions on $\Omega_q$ are continuous and piecewise of class $C^1$, we have that $\mathcal{A}_q$ is non-empty. Furthermore, we have shown that $J$ is coercive on $\mathcal{A}_q$ and convex in the gradient $\nabla(q_1,q_2,q_3)$. Thus, by the direct method in the calculus of variations, we have the existence of a minimizer $(q_1^*,q_2^*,q_3^*)\in\mathcal{A}_q$. We define a function $q_1^s\in\Omega$ by odd reflection of $q_1^*\in\Omega_q$ about the square diagonals and even reflection about $x$- and $y$-axis. An illustration of the reflected solution $q_1^s(x,y)$ is given in Figure \ref{fig:reflected}. We can do the same for the function $q_2^s\in\Omega$ defined by even reflections of $q_2^*$  about the square diagonals and odd reflection about $x$- and $y$-axis and lastly, for the function $q_3^s\in\Omega$ defined by even reflections of $q_3^*$ about the square diagonals and $x$- and $y$-axis.
\begin{figure}
    \centering
    \includegraphics[width=0.4\columnwidth]{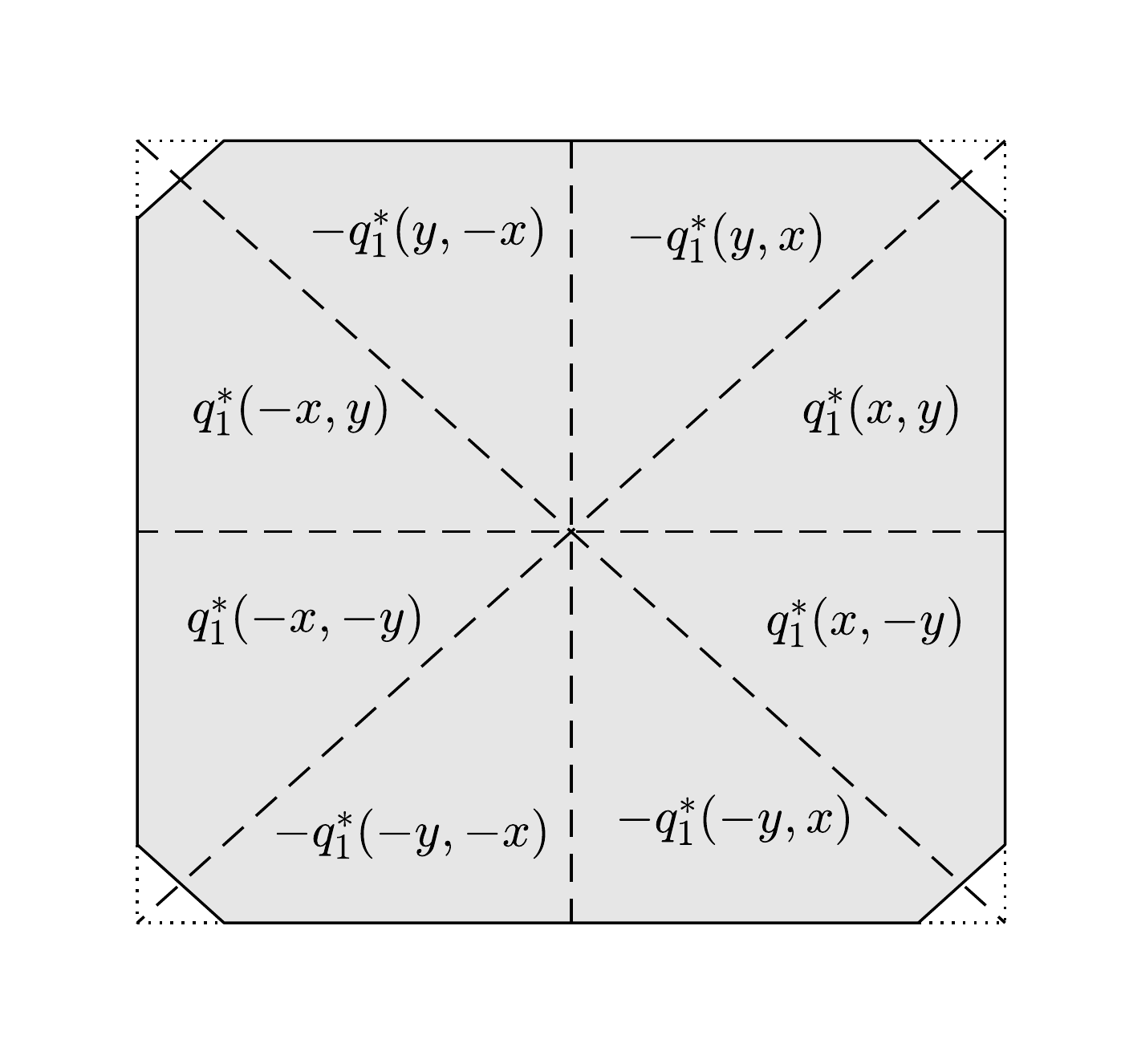}
    \caption{The reflected solution $q_1^s(x,y)$ in Proposition \ref{prop:forever_critical}.}
    \label{fig:reflected}
\end{figure}
 By repeating the arguments in \cite{dangfifepeletier}, we can prove the new triple, $(q_1^s, q_2^s, q_3^s)$, is a weak solution of the associated Euler-Lagrange equation on $\Omega$. One can verify that $(q_1^s,q_2^s,q_3^s)$ is a critical point of $J$ on $\mathcal{A}_0$ with the desired properties.
\end{proof}

\begin{proposition}\label{prop:Ring_non_constant}
For $A<0$ and $L_2\neq0$, the critical solution $(q_1^s, q_2^s, q_3^s)$ constructed in Proposition \ref{prop:forever_critical}, has non-constant $q_2^s$ on $\Omega$, for all $\lambda>0$.
\end{proposition}
\begin{proof}
We proceed by contradiction. Assume that $q_2^s$ is constant on $\Omega$. Recalling the boundary conditions (\ref{bcq2q3}), we necessarily have that $q_2^s\equiv0$ in $\Omega$. Substituting $q_2^s \equiv 0$ into (\ref{q2eq}), we obtain
\begin{gather}
    q_3^s(x,y)=F(x)+G(y),
\end{gather}
for arbitrary real-valued functions $F,G$, with $q_3^s=-s_+/6$ on $\partial\Omega$. Therefore, $q_3^s\equiv-s_+/6$ in $\Omega$. Substituting $q_2^s \equiv 0$ and $q_3^s \equiv -s_+/6$ into \eqref{q1eq} and \eqref{q3eq} yields
\begin{align}
    q_{1,yy}^s+q_{1,xx}^s &= f(q_1^s), \label{fqequation}\\
    q_{1,yy}^s-q_{1,xx}^s &= g(q_1^s) + C_g,\label{Cg}
\end{align}
where 
\begin{align}
    f(q_1^s) &= \frac{4C\lambda^2}{(2+L_2)L}(q_1^s)^3 + \frac{2\lambda^2}{(2+L_2)L}(A-\frac{Bs_+}{3}+\frac{Cs_+^2}{6})q_1^s,\\
    g(q_1^s) &= \frac{2\lambda^2}{LL_2}(B-Cs_+)(q_1^s)^2,\\
    C_g = & -\frac{\lambda^2s_+}{LL_2}(A+\frac{Bs_+}{6}+\frac{Cs_+^2}{6}).
\end{align}
From the reduced PDEs for $q_1^s$, (\ref{fqequation}) and (\ref{Cg}), one can calculate
\begin{equation}
    2(q_{1,xx}^s)_{yy}-2(q_{1,yy}^s)_{xx} =f''(q_1^s)((q_{1,y}^s)^2-(q_{1,x}^s)^2)-g''(q_1^s)((q_{1,y}^s)^2+(q_{1,x}^s)^2)\\+f'(q_1^s)(g(q_1^s)+C_g) - g'(q_1^s)f(q_1^s).\\
   \label{xxyy}
\end{equation}
Furthermore, from the symmetry properties of the constructed solution $q_1^s$ in Proposition \ref{prop:forever_critical}, we have
\begin{equation}\label{q1_00}
    q_1^s\vert_{(0,0)} = q_{1,x}^s\vert_{(0,0)} = q_{1,y}^s\vert_{(0,0)} = 0.
\end{equation}    
Substituting \eqref{q1_00} into \eqref{xxyy}, we obtain
\begin{equation}\label{q1_nonzero}
    (2q_{1,xxyy}^s-2q_{1,yyxx}^s)\vert_{(0,0)} = (f'(q_1^s)C_g)\vert_{(0,0)} = -\frac{2\lambda^2}{(2+L_2)L}(A-\frac{Bs_+}{3}+\frac{Cs_+^2}{6})\frac{\lambda^2s_+}{LL_2}(A+\frac{Bs_+}{6}+\frac{Cs_+^2}{6}).
\end{equation}
If $A\neq-B^2/3C$, then the right hand side of equation (\ref{q1_nonzero}) at $(0,0)$ is non-zero, which leads to a contradiction. If $A=-B^2/3C$, then $q_3^s\equiv-s_+/6=-B/6C$ and (\ref{Cg}) reduces to
\begin{gather}
    q_{1,yy}^s-q_{1,xx}^s=0,
\end{gather}
which implies that $q_1^s$ is of the following form:
\begin{gather}
    q_1^s(x,y)=F_1(x-y)+F_2(x+y),
\end{gather}
for arbitrary real-valued functions $F_1,F_2$. From Proposition \eqref{prop:forever_critical}, we know that for any $\lambda>0$, $q_1^s$ satisfies the symmetry property $q_1^s(x,y)=q_1^s(x,-y)$ and hence,
\begin{gather}
F_1(x-y)+F_2(x+y)=F_1(x+y)+F_2(x-y), \quad (x,y)\in\Omega. \label{bothsides}
\end{gather}
Subtracting $F_2(x-y)+F_2(x+y)$ on both sides of the equality (\ref{bothsides}), we get\begin{gather}
G(x-y)=F_1(x-y)-F_2(x-y)=F_1(x+y)-F_2(x+y)=G(x+y), \quad (x,y)\in\Omega.
\end{gather}
Therefore, 
\begin{gather}
G(z)=F_1(z)-F_2(z)\equiv K, \quad z\in(-2,2),
\end{gather}
for some constant $K$. The function $q_1^s$ may now be rewritten as
\begin{gather}
q_1^s(x,y)=F_1(x+y)+F_1(x-y)-K, \quad (x,y)\in\Omega.   
\end{gather}
This formulation cannot be extended continuously on the boundary since, for $(x,y)=(0,1)$, $(-1,0)$ and $(1,0)$, we have
\begin{gather}
F_1(1)+F_1(-1)-K=\frac{s_+}{2},\ 2F_1(-1)-K=-\frac{s_+}{2},\ 2F_1(1)-K=-\frac{s_+}{2},
\end{gather}
which again leads to the required contradiction.
\end{proof}

\begin{proposition}\label{prop2}
There exists a critical edge length $\lambda_0>0$ such that, for any $\lambda<\lambda_0$, the critical point, $(q_1,q_2,q_3)$, constructed in Proposition \ref{prop:forever_critical} is the unique critical point of the LdG energy (\ref{funcq123}).
\end{proposition}
\begin{proof}
We adapt the uniqueness criterion argument in Lemma 8.2 of \cite{lamy2014}. 
Let $(q_1^{\lambda},q_2^{\lambda},q_3^{\lambda})$ be a global minimizer of energy functional $J$ in \eqref{funcq123} for $\lambda>0$. Let $(q_1^{\infty}(\mathbf{x}),q_2^{\infty}(\mathbf{x}),q_3^{\infty}(\mathbf{x})) \in \mathcal{A}_0$ be such that 
\begin{equation}\label{minfb}
f_b(q_1^{\infty}(\mathbf{x}),q_2^{\infty}(\mathbf{x}),q_3^{\infty}(\mathbf{x}))=\min f_b = \frac{A}{3}s_+^2-\frac{2B}{27}s_+^3+\frac{C}{9}s_+^4,
\end{equation}
a.e. $\mathbf{x}\in\Omega$.
Defining $\bar{f}_b(q_1,q_2,q_3) = \frac{1}{L}(f_b(q_1,q_2,q_3)-\min f_b(q_1,q_2,q_3))$, where $L$ is constant, we have 
\begin{equation}\label{qinfty}
\int_{\Omega}f_{el}(q_1^{\lambda},q_2^{\lambda},q_3^{\lambda})\,\textrm{dA}\leq\int_{\Omega}f_{el}(q_1^{\lambda},q_2^{\lambda},q_3^{\lambda})+\lambda^2\bar{f}_b(q_1^{\lambda},q_2^{\lambda},q_3^{\lambda})\,\textrm{dA}\leq\int_{\Omega}f_{el}(q_1^{\infty},q_2^{\infty},q_3^{\infty})\,\textrm{dA} = M_2(A,B,C,L_2),
\end{equation}
for some constant, $M_2>0$, depending only on $A,B,C$ and $L_2$.
Thus, we restrict ourselves to the following admissible space of $\mathbf{Q}$-tensors:
\begin{equation}
\mathcal{A}_{upper} = \left\{\mathbf{Q}:\int_{\Omega} \frac{1}{2}|\nabla \mathbf{Q}|^2 \textrm{dA}\leq M_2(A,B,C,L_2)\right\}. 
\end{equation}
We note that the second derivatives of $f_b$ are quadratic polynomials in $(q_1, q_2, q_3)$. By an application of the relevant embedding theorem in \cite{brezis2010functional} (Theorem 9.16 which implies that for a bounded domain $\Omega\subset\mathbb{R}^N$ with Lipschitz boundary, for any $u\in C^1_c(\Omega)$, $||u||_{L^p}\leq c||u||_{W^{1,2}}$, $\forall p\in[N,\infty)$, with constant $c$ depending only on $\Omega$), we have that there exists some constant $c_0$, depending only on $A,B,C$ and $\Omega$, such that
\begin{equation}
\left(\int_{\Omega}|f_b''|^2\textrm{dA}\right)^{1/2}\leq c_0(A,B,C,\Omega)\left(\int_{\Omega}|\nabla \mathbf{Q}|^2\,\mathrm{dA}\right)^{1/2} \leq c_0\sqrt{M_2}.
\end{equation} 
We apply the H\"{o}lder inequality to get, for any $x,y\in\mathcal{A}_{upper}$,
\begin{equation}
    \int_{\Omega}f_b\left(\frac{x+y}{2}\right)-\frac{1}{2}f_b(x)-\frac{1}{2}f_b(y)\, \textrm{dA}\leq 
    \frac{1}{8}\sup_{\mathcal{A}_{upper}}\left(\int_{\Omega}|f_b''|^2\textrm{dA}\right)^{1/2} \left(\int_{\Omega}|x-y|^4\textrm{dA}\right)^{1/2} \leq \frac{c_0 \sqrt{M_2}}{8}||x-y||^2_{L_4}
\end{equation}
Therefore, for any $(q_1, q_2, q_3), (\tilde{q_1}, \tilde{q_2}, \tilde{q_3}) \in \mathcal{A}_{upper}$, we have
\begin{equation}\label{finenergy1}
 \int_{\Omega}f_b\left(\frac{q_1+\tilde{q_1}}{2},\frac{q_2+\tilde{q_2}}{2},\frac{q_3+\tilde{q_3}}{2}\right)-\frac{1}{2}f_b(q_1,q_2,q_3)-\frac{1}{2}f_b(\tilde{q_1},\tilde{q_2},\tilde{q_3})\,\mathrm{dA}\leq c_1||q_1-\tilde{q_1},q_2-\tilde{q_2},q_3-\tilde{q_3}||^2_{L_4}
\end{equation}
where $c_1=c_1(A, B, C, L_2, \Omega)>0$. Using (\ref{posl2}), an application of the Poincar\'{e} inequality, and repeating the same arguments as above, we have
 \begin{align}
  \int_{\Omega}f_{el}(q_1-\tilde{q_1},q_2-\tilde{q_2},q_3-\tilde{q_3})\,\mathrm{dA}
  &\geq  \,\min\{1,1+L_2\}\int_{\Omega}|\nabla(q_1-\tilde{q}_1)|^2+|\nabla(q_2-\tilde{q}_2)|^2+3|\nabla(q_3-\tilde{q_3})|^2\,\mathrm{dA} \\
  &\geq  \min\{1,1+L_2\}K(\Omega)\left(||q_1-\tilde{q}_1||^2_{W^{1,2}}+||q_2-\tilde{q}_2||^2_{W^{1,2}}+3||q_3-\tilde{q_3}||^2_{W^{1,2}}\right)\\
  &\geq  \,\, c_2(\Omega,L_2)||q_1-\tilde{q_1},q_2-\tilde{q_2},q_3-\tilde{q_3}||^2_{L^4} \label{finenergy2}
 \end{align}
for some constant, $c_2>0$, depending only on $\Omega$ and the sign of $L_2$. Using both (\ref{finenergy1}) and (\ref{finenergy2}), we have
\begin{align}
     J\left[\frac{q_1+\tilde{q_1}}{2},\frac{q_2+\tilde{q_2}}{2},\frac{q_3+\tilde{q_3}}{2}\right]\leq&\,\,\frac{1}{2}J[q_1,q_2,q_3]+\frac{1}{2}J[\tilde{q_1},\tilde{q_2},\tilde{q_3}]-\frac{c_2}{4}||q_1-\tilde{q_1},q_2-\tilde{q_2},q_3-\tilde{q_3}||^2_{L^4}\nonumber \\
     &\qquad+\frac{c_1\lambda^2}{L}||q_1-\tilde{q_1},q_2-\tilde{q_2},q_3-\tilde{q_3}||^2_{L^4} \\
     =&\,\,\frac{1}{2}J[q_1,q_2,q_3]+\frac{1}{2}J[\tilde{q_1},\tilde{q_2},\tilde{q_3}]-\frac{c_2}{8}||q_1-\tilde{q_1},q_2-\tilde{q_2},q_3-\tilde{q_3}||^2_{L^4}\nonumber \\
     &\qquad-c_1\left(\frac{c_2}{8c_1}-\frac{\lambda^2}{L}\right)||q_1-\tilde{q_1},q_2-\tilde{q_2},q_3-\tilde{q_3}||^2_{L^4}, \nonumber
     \\
      \leq \,\,&\frac{1}{2}J[q_1,q_2,q_3]+\frac{1}{2}J[\tilde{q_1},\tilde{q_2},\tilde{q_3}]
 \end{align}
for $\lambda\leq\lambda_0:=\sqrt{c_2L/(8c_1)}$. Thus, $J$ is strictly convex for the finite energy triplets $(q_1,q_2,q_3)$, and has a unique critical point for $\lambda<\lambda_0$. We deduce that the critical point constructed in Proposition \ref{prop:forever_critical} is the unique minimizer of $J[q_1, q_2, q_3]$ and, in fact, the unique global LdG energy minimizer (when we consider $\Qvec$-tensors with the full five degrees of freedom, as opposed to this reduced setting, (\ref{q123}), with three degrees of freedom), for sufficiently small $\lambda$.
\end{proof}
\begin{lemma}\label{lem1}
Suppose that $(q_1,q_2,q_3)$ is the unique global minimizer of the energy (\ref{funcq123}), for $\lambda<\lambda_0$ given by Proposition \ref{prop2}. Then for any $L_2>-1$, the function $q_1:\Omega\to\mathbb{R}$ vanishes along the square diagonals $y=x$ and $y=-x$, and the function $q_2:\Omega\to\mathbb{R}$ vanishes along $y=0$ and $x=0$.
\end{lemma}
\begin{proof} This is in fact an immediate consequence of Proposition~\ref{prop:forever_critical}, but we present an alternative short proof based on symmetry observations.
Suppose that $(q_1,q_2,q_3)\in W^{1,2}(\Omega,\mathbb{R}^3)$ is a global minimizer of the associated energy functional $J$, in the admissible class $\mathcal{A}_0$ for a given $\lambda>0$. Then $(q_1(x,y),q_2(x,y),q_3(x,y))$ is a solution of the Euler-Lagrange system (\ref{q1eq})--(\ref{q3eq}), subject to the boundary conditions (\ref{bcq1}) and (\ref{bcq2q3}). So are the triples
\begin{gather*}
    (q_1(-x,y),-q_2(-x,y),q_3(-x,y)),\quad (q_1(x,-y),-q_2(x,-y),q_3(x,-y)),\quad (-q_1(y,x),q_2(y,x),q_3(y,x))
\end{gather*}
that are compatible with the imposed boundary conditions. 
 We combine this symmetry result with the uniqueness result in Proposition \ref{prop2} to get the desired conclusion. For example, we simply use $q_1(x,y)=-q_1(y,x)$ with $x=y$ to deduce that $q_1(x,x)=0$. Also, $q_1(-x,y)=q_1(x,y)$ with $x=y$ yields that $q_1(x,-x)=q_1(x,x)=0$. Furthermore, we use the relation $q_2(x,y)=-q_2(-x,y)$ with $x=0$ to deduce that $q_2(0,y)=0$, and similarly, $q_2(x,y)=-q_2(x,-y)$ with $y=0$ to deduce that $q_2(x,0)=0$.
\end{proof}

\textbf{Remark:} Let $(q_1^*, q_2^*, q_3^*)$ be a critical point of $J$ in $\mathcal{A}_0$. If $q_3^*$ is a constant, then we necessarily have $q_3^*\equiv-s_+/6$ in $\Omega$. Therefore, equations (\ref{q1eq})--(\ref{q2eq}) become
\begin{align}
   \left(1+\frac{L_2}{2}\right)\Delta q_1&=\frac{\lambda^2}{L}q_1\left(A-\frac{Bs_+}{3}+2C(q_1^2+q_2^2)+\frac{Cs_{+}^2}{6}\right),\label{laplace_q1}\\
   \left(1+\frac{L_2}{2}\right)\Delta q_2&=\frac{\lambda^2}{L}q_2\left(A-\frac{Bs_+}{3}+2C(q_1^2+q_2^2)+\frac{Cs_{+}^2}{6}\right).
\end{align}
From Proposition \ref{prop2}, we know that for $\lambda<\lambda_0(A,B,C,L,L_2)$, the solution is unique, and hence we have $q_2^*\equiv0$ in $\Omega$.
Analogous to the proof in Proposition \ref{prop:Ring_non_constant}, we get a contradiction and deduce that $q_2^*$ and $q_3^*$ are non-constant throughout $\Omega$, for $\lambda$ small enough.
\begin{figure}[htbp]
    \begin{center}
    \includegraphics[width=0.85\columnwidth]{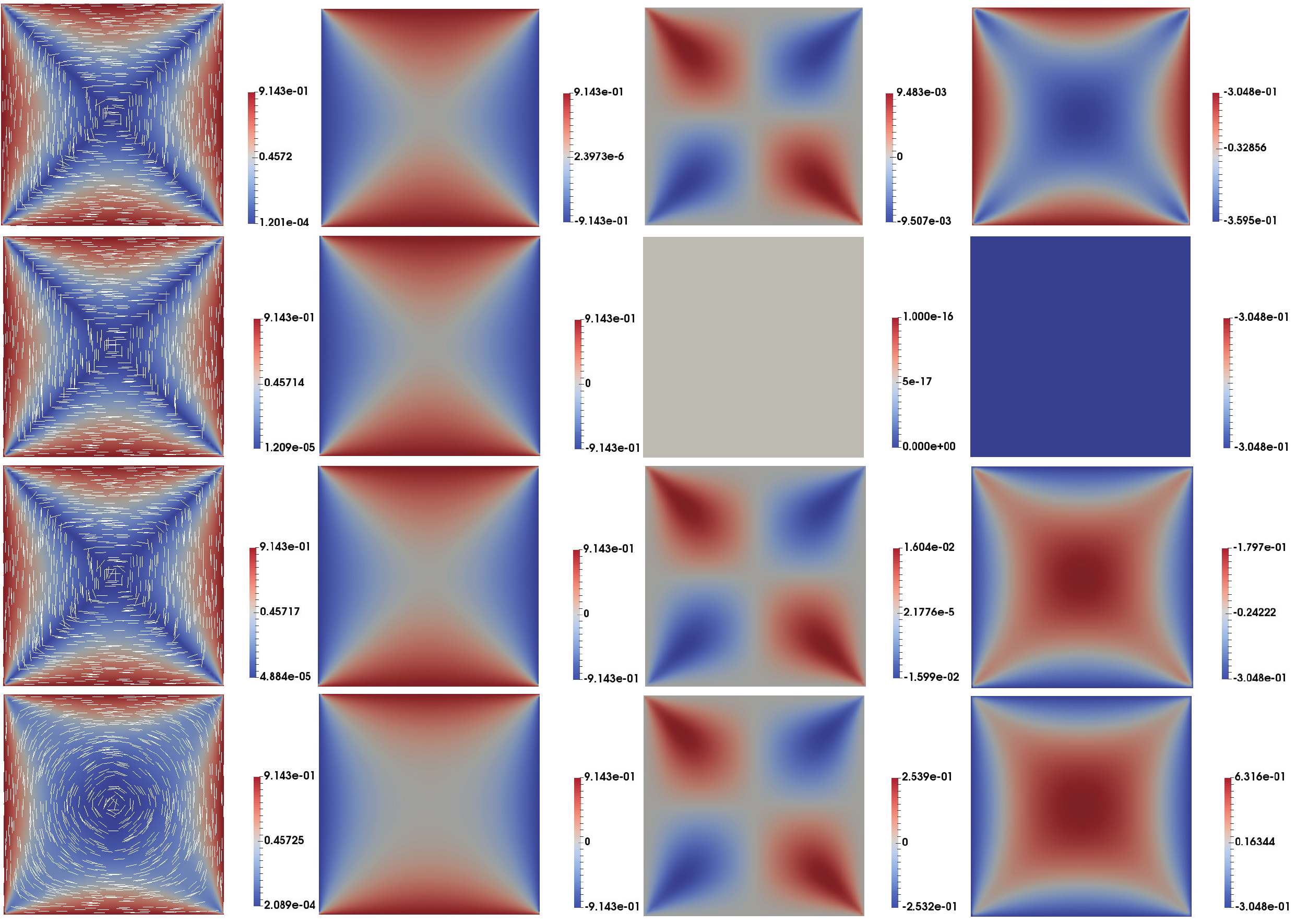}
        \caption{The unique stable solution of the Euler-Lagrange equations \eqref{q1eq}--\eqref{q3eq}, with $\bar{\lambda}^2 = 5$, and (from the first to fourth row) $L_2 = -0.5$, $0$, $1$ and $10$, respectively. In the first column, we plot the scalar order parameter $s^2 = q_1^2 + q_2^2$ by color from blue to red, and the director profile $\mathbf{n} = (\cos(arctan(q_2/q_1)/2),\sin(arctan(q_2/q_1)/2))$ in terms of white lines. We plot the corresponding $q_1,q_2$ and $q_3$ profiles, in the second to fourth columns, respectively.}
        \label{fig:RING_lambda_5}
    \end{center}
\end{figure}

As in \cite{wang2019order,canevariharrismajumdarwang}, we frequently refer to the following dimensionless parameter in our numerical simulations:
\begin{gather*}
    \bar{\lambda}^2:=\frac{2C\lambda^2}{L}.
\end{gather*}
In Figure \ref{fig:RING_lambda_5}, we plot the unique stable solution of \eqref{q1eq}--\eqref{q3eq} with $\bar{\lambda}^2= 5$, for $L_2 =-0.5$, $0$, $1$, $10$. In this figure, and all subsequent figures in this paper, we fix $A=-B^2/3C$ with $B=0.64\times10^{4}\,\mathrm{Nm}^{-2}$ and $C=0.35\times10^{4}\,\mathrm{Nm}^{-2}$. When $L_2 = 0$, the solution is the $WORS$ defined by (\ref{WQ}). When $L_2=-0.5$, $1$, and $10$, $q_2$ and $q_3$ are non-constant as proven above. One can check that $q_1:\Omega\to\mathbb{R}$ vanishes along the square diagonals $y=x$ and $y=-x$, and the function $q_2:\Omega\to\mathbb{R}$ vanishes along $y=0$ and $x=0$, as proven in Lemma \ref{lem1}. When $L_2=-0.5,1,$ and $10$, we observe a central $+1$-point defect in the profile of $(q_1,q_2)$, and we label this as the $Ring^+$ solution. 
\begin{figure}
    \begin{center}
    \includegraphics[width=0.85\columnwidth]{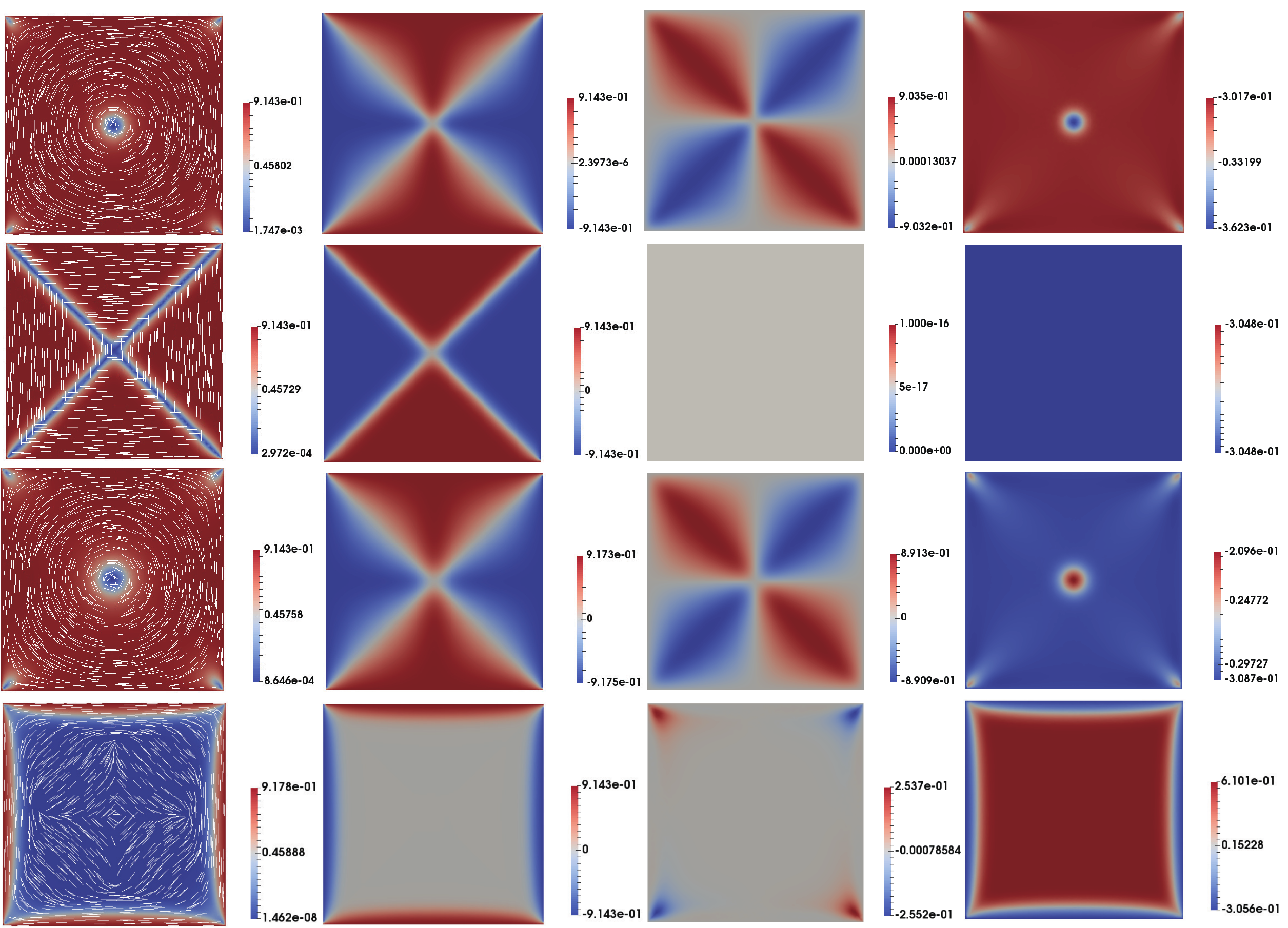}
        \caption{A solution branch for the Euler-Lagrange system \eqref{q1eq}--\eqref{q3eq} with $\bar{\lambda}^2 = 500$, and $L_2 =-0.5$, $0$, $1$ and $10$ respectively, plotted in the first, second, third and fourth row respectively. This solution branch is a symmetric solution branch, as described in Proposition~\ref{prop:forever_critical}. When $L_2 = -0.5$, $0$ and $1$, the plotted solution is unstable. When $L_2 = 10$, the plotted solution is stable. The first column contains plots of the order parameter, $s^2$, and the director $\mathbf{n}$. In the second to fourth column, we plot the corresponding, $q_1,q_2$ and $q_3$ profiles.}
        \label{fig:RING_lambda_500}
    \end{center}
\end{figure}
We perform a parameter sweep of $\bar{\lambda}^2$, from $5$ to $500$, and find one of the symmetric solution branches constructed in Proposition \ref{prop:forever_critical}, such that $q_1:\Omega\to\mathbb{R}$ vanishes along the square diagonals $y=x$ and $y=-x$ and  $q_2:\Omega\to\mathbb{R}$ vanishes along $y=0$ and $x=0$. This branch is a continuation of the $Ring^+$ branch. The solutions with $\bar{\lambda}^2=500$ are plotted in Figure \ref{fig:RING_lambda_500}. When $L_2=0$, we find the $WORS$ for all $\lambda>0$. When $-1<L_2<0$, the solution exhibits a $+1$-defect at the square center, continued from the $Ring^+$ branch and hence, we refer to it as the $Ring^+$ solution.  When $L_2$ is positive and moderate in value, we recover the $Ring^+$ solution branch and and the corresponding $q_3<-s_+/6$ at the square center for negative $L_2$, but $q_3>-s_+/6$ for positive $L_2$. When $L_2$ is large enough, we recover a symmetric solution which is approximately constant,  $(0,0,s_+/3)$, away from the square edges, shown in the fourth row of Figure \ref{fig:RING_lambda_500} for $L_2=10$. We refer to this solution as the \emph{Constant} solution in the remainder of this manuscript.

\section{Asymptotic Studies}
\label{sec:asymptotics}
\subsection{ The Small $\lambda$ and Small Anisotropy $L_2\to 0$ limit}\label{sec:smallL2smalllambda}
We work at the special temperature $A=-B^2/3C$ to facilitate comparison with the results in \cite{hansiap2020}, where the authors investigate solution landscapes with $L_2 = 0.$ Notably, for $L_2=0$ and $A = -B^2/3C$, reduced LdG solutions on 2D polygons have $q_3 \equiv -s_+/6=-B/6C$ for our choice of TBCs, and it is natural to investigate the effects of the anisotropy parameter, $L_2$, on these reduced solutions in 2D frameworks. 
At $A=-B^2/3C$, the governing Euler-Lagrange equations are given by the following system of partial differential equations:
\begin{align}
\left(1+\frac{L_2}{2}\right)\Delta q_1+\frac{L_2}{2}(q_{3,yy}-q_{3,xx})&=\frac{\lambda^2}{L}q_1\left(-\frac{B^2}{3C}+2Bq_3+2C(q_1^2+q_2^2+3q_3^2)\right), \\
    \left(1+\frac{L_2}{2}\right)\Delta q_2-L_2q_{3,xy}&=\frac{\lambda^2}{L}q_2\left(-\frac{B^2}{3C}+2Bq_3+2C(q_1^2+q_2^2+3q_3^2)\right), \\
    \left(1+\frac{L_2}{6}\right)\Delta q_3+\frac{L_2}{6}(q_{1,yy}-q_{1,xx})-\frac{L_2}{3}q_{2,xy}&=\frac{\lambda^2}{L}q_3\left(-\frac{B^2}{3C}-Bq_3+2C(q_1^2+q_2^2+3q_3^2)\right)+\frac{\lambda^2B}{3L}(q_1^2+q_2^2),
    \label{q1q2q3}
\end{align}
satisfying $q_1=q_{b}$, $q_2=0$, and $q_3=-B/6C$ on $\pp\Omega$. We take a regular perturbation expansion of these functions in the $L_2\to0$ limit. The leading order approximation is given by the $WORS$, $(q,0,-B/6C)$, where $q$ is a solution of the Allen-Cahn equation, as in \cite{canevari2017order}:
\begin{gather}
\Delta q=\frac{2C\lambda^2}{L}q\left(q^2-\frac{B^2}{4C^2}\right), \qquad q=q_{b}\quad\text{on}\quad\pp\Omega.  \label{WORS}
\end{gather}
We may assume that $q_1, q_2, q_3$ are expanded in powers of $L_2$ as follows:
\begin{equation}
\begin{aligned}\label{L2_0_asymptotics}
q_1(x,y)&=q(x,y)+L_2f(x,y)+\dots \\
q_2(x,y)&=L_2g(x,y)+\dots  \\
q_3(x,y)&=-\frac{B}{6C}+L_2h(x,y)+\dots     
\end{aligned}
\end{equation}
for some functions $f, g, h$ which vanish on the boundary. Up to $\mathcal{O}(L_2)$, the governing partial differential equations for $f,g,h$ are given by:
\begin{align}
\Delta f&=\frac{C\lambda^2}{L}\left(4q^2f+(2f-q)\left(q^2-\frac{B^2}{4C^2}\right)\right),    \label{feq}\\
\Delta g&=\frac{2C\lambda^2}{L}g\left(q^2-\frac{B^2}{4C^2}\right),\label{geq}\\
\Delta h&=\frac{2C\lambda^2}{L}h\left(q^2+\frac{B^2}{4C^2}\right)-\frac{1}{6}(q_{,yy}-q_{,xx}), \label{heq}
\end{align}
where $q=q_{b}, f=g=h=0$ on $\pp\Omega$. One can easily verify that the system (\ref{feq})--(\ref{heq}), are the Euler-Lagrange equations for the following functional:
\begin{align}
    F(f,g,h):=&\int_{\Omega}\left\{|\nabla f|^2+|\nabla g|^2+|\nabla h|^2+\frac{1}{3}(q_{,y}h_{,y}-q_{,x}h_{,x})\right\}\,\mathrm{dA} \nonumber \\
    &+\frac{2C\lambda^2}{L}\int_{\Omega}\left\{(f^2+g^2)(q^2-\frac{B^2}{4C^2})+h^2(q^2+\frac{B^2}{4C^2})-fq(q^2-\frac{B^2}{4C^2}-2fq)\right\}\,\mathrm{dA}.
\end{align}
For $\lambda$ small enough, one can show that there exists a unique solution $(f,g,h)\in W^{1,2}_0(\Omega;\mathbb{R}^3)$ of the system (\ref{feq})--(\ref{heq}), by following the approach in \cite{lamy2014} to show that $F$ is strictly convex in $W^{1,2}_0(\Omega;\mathbb{R}^3)$ for sufficiently small $\lambda$. Hence, for $\lambda$ small enough, we have $g\equiv0$ on $\Omega$.
Similarly to Proposition \ref{lem1}, we can check that if $(q(x,y),f(x,y),g(x,y),h(x,y))$ is a solution of \eqref{WORS},\eqref{feq}--\eqref{heq}, then the quadruplets, $(-q(y,x),-f(y,x),g(y,x),h(y,x))$ and $(q(-x,y),f(-x,y),g(-x,y),h(-x,y))$, are also solutions of \eqref{WORS},\eqref{feq}--\eqref{heq}. Thus we have $f(x,y) = 0$ on diagonals for $\lambda$ small enough. Hence, for $\lambda$ small enough, the cross structure of the $WORS$ is lost mainly because of effects of $L_2$ on the component $q_3$, as we discuss below.

 From \cite{fang2020surface}, the solutions of \eqref{WORS} with $\lambda=0$, are a good approximation to the solutions of \eqref{WORS} for sufficiently small $\lambda$. When $\lambda = 0$, $q = q_0$ where
 \begin{gather}\label{q0eq}
    \Delta q_0 = 0,\quad (x,y)\in\Omega, \qquad q_0=q_{1b},\quad\text{on}\quad\pp\Omega
\end{gather}
The analytical solution of \eqref{q0eq} is given by \cite{lewis2015defects}:
\begin{align}\label{q0so}
    q_0(x,y) &= \frac{s_+}{2}\sum_{k\ odd}\frac{4}{k\pi}\left(\sin\left(\frac{k\pi(x+1)}{2}\right)\frac{\sinh(k\pi(1-y)/2)+\sinh(k\pi(1+y)/2)}{\sinh(k\pi)}\right. \nonumber \\
    &\left.-\sin\left(\frac{k\pi(y+1)}{2}\right)\frac{\sinh(k\pi(1-x)/2)+\sinh(k\pi(1+x)/2)}{\sinh(k\pi)}\right).
\end{align}
 When $\lambda = 0$, the unique solution of \eqref{feq}--\eqref{heq} is $f = f_0\equiv 0$, $g = g_0\equiv 0$ and $h = h_0$ where
\begin{equation}\label{h0eq}
    \Delta h_0 = -\frac{1}{6}(q_{0,yy}-q_{0,xx}),
\end{equation}
with $h_0=0$ on $\pp\Omega$. In Figure \ref{fig:L2_0_asymptotic_compare}, we plot the difference between the solution, $(q_1,q_2,q_3)$ of \eqref{q1q2q3} and the triplet, $(q_0+L_2f_0,L_2g_0,-\frac{s_+}{6}+L_2h_0) = (q_0, 0 ,-\frac{s_+}{6}+L_2h_0)$, which is the solution of (\ref{feq})-(\ref{heq}) for $\lambda = 0$.

\begin{proposition}
The analytical solution of \eqref{h0eq} is given by
\begin{equation}
h_0(x,y)=\sum_{m,n\ odd}\frac{16s_+mn}{3\pi^2(m^2+n^2)^2}\sin\left(\frac{m\pi(x+1)}{2}\right)\sin\left(\frac{n\pi(y+1)}{2}\right), \label{eq:h0analytic}
\end{equation}
where $h_0(0,0)$ is positive.
\end{proposition}
\begin{proof}
Firstly, we calculate the analytical solution of \eqref{h0eq}.
Substituting \eqref{q0so} into \eqref{h0eq}, we have
\begin{align}
    \Delta h_0 &=-\sum_{k\ odd}\frac{s_+k\pi}{6\sinh(k\pi)}\left\{\sin\left(\frac{k\pi(x+1)}{2}\right)\left(\sinh\left(\frac{k\pi(1-y)}{2}\right)+\sinh\left(\frac{k\pi(1+y)}{2}\right)\right)\right.\nonumber\\
    &\left.+\sin\left(\frac{k\pi(y+1)}{2}\right)\left(\sinh\left(\frac{k\pi(1-x)}{2}\right)+\sinh\left(\frac{k\pi(1+x)}{2}\right)\right)\right\}, \label{hexpand}
\end{align}
which is a homogeneous Poisson equation. We consider the transformations $\hat{x}=x+1$ and $\hat{y}=y+1$, such that $h_0(\hat{x},\hat{y}):[0,2]^2\to\mathbb{R}$. We apply the method of eigenfunction expansions
\begin{gather}
    h_0(\hat{x},\hat{y})=\sum_{n=1}^{\infty}\sum_{m=1}^{\infty}E_{mn}\sin\left(\frac{m\pi\hat{x}}{2}\right)\sin\left(\frac{n\pi\hat{y}}{2}\right),
\end{gather}
where $E_{mn}$ are double Fourier sine series coefficients.
Following standard Fourier series-type calculations, we obtain:
\begin{gather}
    h_0(\hat{x},\hat{y})=\sum_{m,n\ odd}\frac{16s_+mn}{3\pi^2(m^2+n^2)^2}\sin\left(\frac{m\pi\hat{x}}{2}\right)\sin\left(\frac{n\pi\hat{y}}{2}\right),
\end{gather}
which yields the form (\ref{eq:h0analytic}). Substituting $(x,y) = (0,0)$ i.e., $(\hat{x},\hat{y}) = (1,1)$ into \eqref{eq:h0analytic}, we obtain
\begin{align}
 h_0(1,1) &= \frac{16s_+}{3\pi^2}\sum_{m,n\ odd}\frac{mn}{(m^2+n^2)^2}\sin\left(\frac{m\pi}{2}\right)\sin\left(\frac{n\pi}{2}\right)\nonumber \\
 &=\frac{16s_+}{3\pi^2}\sum_{m,n\ odd}\frac{mn}{(m^2+n^2)^2}(-1)^{\frac{m+n-2}{2}} \nonumber \\
 &=\frac{16s_+}{3\pi^2}\sum_{m\ odd}\frac{1}{4m^2} + \frac{32s_+}{3\pi^2} \sum_{m>n\ odd}\frac{mn}{(m^2+n^2)^2}(-1)^{\frac{m+n-2}{2}} \nonumber \\
  &=\frac{s_+}{6}+  \frac{32s_+}{3\pi^2} \sum_{m>n\ odd}\frac{mn}{(m^2+n^2)^2}(-1)^{\frac{m+n-2}{2}}\label{h000}
\end{align}
The summed term in \eqref{h000} is
\begin{align}
 \sum_{m>n\ odd}\frac{mn}{(m^2+n^2)^2}(-1)^{\frac{m+n-2}{2}}&= \sum_{n\ odd,m = n+4s-2}\frac{mn}{(m^2+n^2)^2}(-1)^{\frac{m+n-2}{2}} \nonumber \\&+ \sum_{n\ odd,m = n+4s}\frac{mn}{(m^2+n^2)^2}(-1)^{\frac{m+n-2}{2}} \nonumber \\
&\geq  \sum_{n\ odd,m = n+4s-2}\frac{mn}{(m^2+n^2)^2}(-1)^{\frac{m+n-2}{2}}\label{mn}
\end{align}
Setting $k = m+n$, $l = m-n$, we have
$mn/(m^2+n^2)^2 = (k^2-l^2)/(k^2+l^2)^2$. Since $n$ is odd, we have $k = m+n = 2n+4s-2 = 2(2p+1)+4s-2 = 4(p+s) = 4r$, $(p\geq0,s,r\geq1)$, and $l = m-n = 4s-2$. Substituting $k = m+n$, $l = m-n$ into \eqref{mn}, we obtain
\begin{align}
 \sum_{n\ odd,m = n+4s-2}\frac{mn}{(m^2+n^2)^2}(-1)^{\frac{m+n-2}{2}}&=\sum_{k>l,k = 4r,l=4s-2}\frac{k^2-l^2}{(k^2+l^2)^2}(-1)^{\frac{k-2}{2}} \nonumber \\
&=-\sum_{k>l,k = 4r,l=4s-2}\frac{k^2-l^2}{(k^2+l^2)^2} \nonumber\\
&\geq -\sum_{k = 4r}\frac{1}{k^2} = -\sum_{r = 1}^{\infty}\frac{1}{16r^2} = -\frac{\pi^2}{96}
\end{align}
Hence, from \eqref{h000} we have $h_0(0,0)\geq\frac{s_+}{6}-\frac{s_+}{9}>0$.
\end{proof}
\begin{figure}[htbp]
    \begin{center}
    \includegraphics[width=0.7\columnwidth]{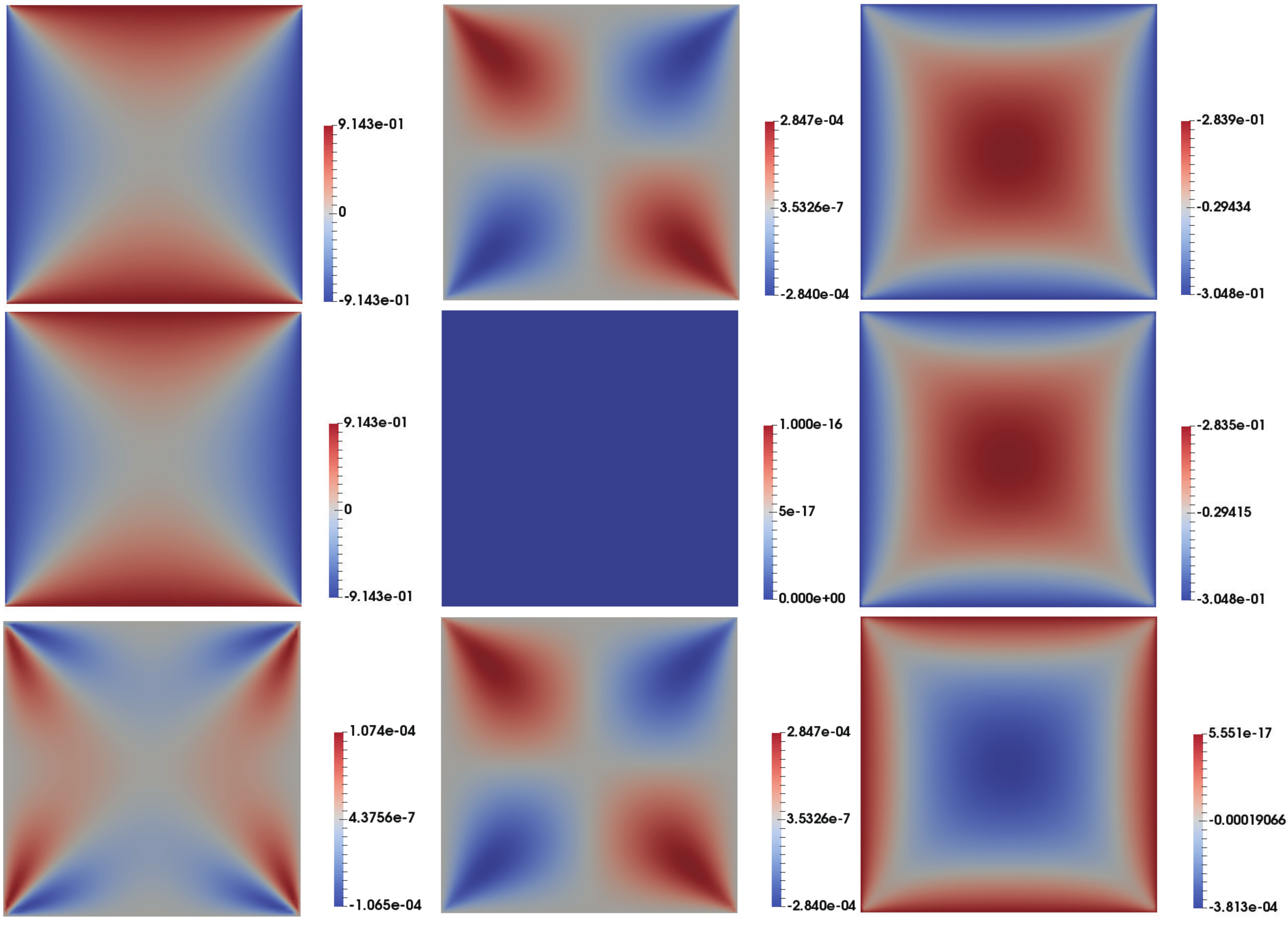}
        \caption{First row: plots of $(q_1,q_2,q_3)$, the solution of \eqref{q1eq}--\eqref{q3eq} with $\bar{\lambda}^2 = 0.01$, $L_2 = 0.1$. Second row: plots of asymptotic solution $(q_0+L_2f_0,L_2g_0,-\frac{s_+}{6}+L_2h_0) = (q_0,0,-\frac{s_+}{6}+L_2h_0)$, where $q_0$ and $h_0$ are the solutions of \eqref{q0eq} and \eqref{h0eq}, respectively. Third row: plots of the difference between solution in first and second row, $(q_1,q_2,q_3)-(q_0+L_2f_0,L_2g_0,-\frac{s_+}{6}+L_2h_0)$. } 
        \label{fig:L2_0_asymptotic_compare}
    \end{center}
\end{figure}


\subsection{The $L_2\to+\infty$ limit}

Consider the regular perturbation expansion in powers of $1/L_2$ of the solutions, $q_1, q_2, q_3$, of the Euler-Lagrange system (\ref{q1eq})--(\ref{q3eq}) subject to the boundary conditions (\ref{bcq1}) and (\ref{bcq2q3}). Then the governing Euler-Lagrange equations (\ref{q1eq})--(\ref{q3eq}) to leading order that is, $\mathcal{O}(L_2)$, are given by the following system of PDEs:
\begin{align}
\frac{1}{2}\Delta \rho+\frac{1}{2}(\tau_{yy}-\tau_{xx})&=0, \label{inf1} \\
\frac{1}{2}\Delta \sigma-\tau_{xy}&=0, \label{inf2}\\
\frac{1}{6}\Delta \tau+\frac{1}{6}(\rho_{yy}-\rho_{xx})-\frac{1}{3}\sigma_{xy}&=0, \label{inf3}
\end{align}
where the $\rho, \sigma, \tau$ are the leading order approximations of $q_1,q_2,q_3$, respectively in the $L_2\to\infty$ limit.
\begin{proposition}
The leading order system of Euler-Lagrange equations in the $L_2\to\infty$ limit, (\ref{inf1})--(\ref{inf3}), is not an elliptic PDE system.
\end{proposition}
\begin{proof}
The system of equations (\ref{inf1})--(\ref{inf3}) can be written as
\begin{gather*}
    A\mathbf{q}_{0,xx}+2B\mathbf{q}_{0,xy}+C\mathbf{q}_{0,yy}=\mathbf{0},
\end{gather*}
where $\mathbf{q}_0=(\rho,\sigma,\tau)$ and
\begin{gather*}
    A=\begin{pmatrix}
    \frac{1}{2} & 0 & -\frac{1}{2} \\
    0 & \frac{1}{2} & 0 \\
    -\frac{1}{6} & 0 & \frac{1}{6}
    \end{pmatrix},
    \quad B=\begin{pmatrix}
    0 & 0 & 0 \\
    0 & 0 & -\frac{1}{2} \\
    0 & -\frac{1}{6} & 0
    \end{pmatrix},
   \quad C=\begin{pmatrix}
    \frac{1}{2} & 0 & \frac{1}{2} \\
    0 & \frac{1}{2} & 0 \\
    \frac{1}{6} & 0 & \frac{1}{6}
    \end{pmatrix}.
\end{gather*}
The system is said to be \textit{elliptic}, in the sense of I.G. Petrovsky \cite{Gu}, if the determinant
\begin{gather*}
    |A\alpha^2+2B\alpha\beta+C\beta^2|\neq0,
\end{gather*}
for any real numbers $\alpha,\beta \neq 0$. We can check that for this system, we have
\begin{gather*}
    |A\alpha^2+2B\alpha\beta+C\beta^2|\equiv 0.
\end{gather*}
for any real numbers $\alpha,\beta$.
Hence, the limiting problem (\ref{inf1})--(\ref{inf3}) is not an elliptic problem.
\end{proof}

\begin{proposition}\label{no_solution}
    There is no classical solution of the limiting problem (\ref{inf1})--(\ref{inf3}), with the boundary conditions (\ref{bcq1}) in the $\epsilon\to 0$ limit and (\ref{bcq2q3}), where $\epsilon$ is the short edge length of the truncated unit square, $\Omega$.
\end{proposition}
\begin{proof}
As $L_2\to\infty$, the minimizers $(q_1,q_2,q_3)$ of the energy $J$ in \eqref{funcq123}, with $f_{el}$ as in \eqref{pos}, are constrained to satisfy \begin{gather*}
    f_{div}(q_1,q_2,q_3)=(q_{1,x}+q_{2,y}-q_{3,x})^2+(q_{2,x}-q_{1,y}-q_{3,y})^2=0,\quad\text{a.e.}\quad(x,y)\in\Omega,
\end{gather*}
subject to the Dirichlet TBCs (\ref{bcq1}) and (\ref{bcq2q3}). Up to $\mathcal{O}(L_2)$, this corresponds to the following PDEs for the leading order approximations $\rho,\sigma,\tau$:
\begin{align}
    &(\rho-\tau)_{x}+\sigma_{y}=0,\label{hopf1}\\
    &\sigma_{x}-(\rho+\tau)_{y}=0, \label{hopf2}
\end{align} 
almost everywhere, subject to the same TBCs, $\rho=q_{b}$, $\sigma=0$, $\tau =-s_+/6$ on $\pp\Omega$. As $\epsilon\to0$, the boundary conditions for $\rho,\sigma,\tau$ are piecewise constant, and hence the tangential derivatives of $\rho,\sigma$ and $\tau$ vanish on the long square edges.
On $y = \pm 1$, the tangential derivative $(\rho-\tau)_{x}=0$, hence we obtain $\sigma_{y} = 0$ in \eqref{hopf1}.
Similarly, we have $\sigma_{x} = 0$ on $x = \pm 1$.
This implies that
$\pp_{\nu} \sigma=0$ on $\pp\Omega$, where $\pp_{\nu}$ is the outward pointing normal derivative, and we view the equation (\ref{inf2}) to be of the form
\begin{gather*}
\Delta \sigma=u(x,y),\qquad \left.\sigma\right|_{\pp\Omega}=0.    
\end{gather*}
By the Hopf Lemma, when $\pp_{\nu}=0$, we have $\sigma\equiv0$. Following the same arguments as in Proposition \ref{prop:Ring_non_constant}, this requires that $\tau\equiv-s_+/6$. Substituting $\tau\equiv-s_+/6$ into equations (\ref{hopf1}) and (\ref{hopf2}), we obtain $\rho_{x}=\rho_{y}=0$, contradicting the boundary condition (\ref{bcq1}). Hence, there are no classical solutions of the system (\ref{inf1})--(\ref{inf3}).
\end{proof}

Although there is no classical solution of \eqref{inf1}--\eqref{inf3} subject to the imposed boundary conditions, we can use finite difference methods to calculate a numerical solution, see Figure \ref{fig:L2_inf}. We label this solution $(\rho,\sigma,\tau)\equiv(0,0,s_+/3)$ on $\Omega$ as the \emph{Constant} solution, where $\rho$ and $\tau$ are discontinuous on $\partial\Omega$.
\begin{figure}[htb!]
    \begin{center}
     \includegraphics[width=\columnwidth]{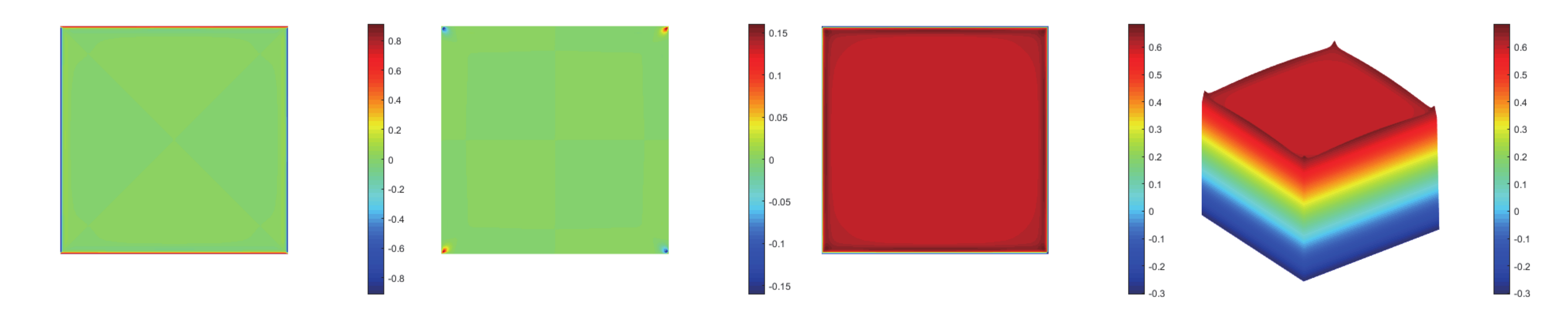}
         \caption{Solutions, $\rho,\sigma,\tau$, of the leading order Euler-Lagrange system \eqref{inf1}--\eqref{inf3} in the $L_2\to\infty$ limit. Plots of (from left to right) $\rho,\sigma,\tau$, and $(x,y,\tau)$.}
         \label{fig:L2_inf}
    \end{center}   
\end{figure}

We now give a heuristic argument to explain the emergence of the Constant solution in the interior of $\Omega$, as $L_2\to\infty$. Assuming $\rho=a$, $\sigma=b$, and $\tau=c$, where $a,b,c\in\mathbb{R}$ are constants, we have $f_{div}=0$ in $\Omega$ up to $\mathcal{O}(L_2)$. On the boundary, using finite difference methods, the first derivatives e.g., $\rho_{x}$, are calculated from the difference between the interior value and the value on the boundary i.e., $\rho_{x}\vert_{x = -1} = (\rho\vert_{interior}-\rho\vert_{boundary})/h = (a-(-s_+/2))/h$, where $h$ is the size of square mesh.
We compute the choices of $a,b$, and $c$, that ensure $f_{div} = 0$ on the boundary, below:
\begin{align*}
(a+\frac{s_+}{2}-c-\frac{s_+}{6})^2 + b^2 &= 0\quad \text{on}\,\, x = -1,\\
(-a-\frac{s_+}{2}+c+\frac{s_+}{6})^2 + (-b)^2 &= 0\quad \text{on}\,\, x = 1,\\
b^2 + (-a+\frac{s_+}{2}-c-\frac{s_+}{6})^2 &= 0\quad \text{on}\,\, y = -1,\\
(-b)^2 + (a-\frac{s_+}{2}+c+\frac{s_+}{6})^2 &= 0\quad \text{on}\,\, y = 1,
\end{align*}
and hence $a = b = 0$, $c = s_+/3$.
Therefore, $\rho = \sigma=0$, $\tau =s_+/3$ is the unique stable solution of \eqref{inf1}--\eqref{inf2}, except for zero measure sets and we label $(q_1,q_2,q_3) = (0,0,s_+/3)$ as the physically relevant Constant solution in the $L_2 \to \infty$ limit. This is consistent with the numerical results in Figure \ref{fig:RING_lambda_5}. 

\subsection{The $\lambda\to\infty$ limit}
\label{sec:lambdainfty}
The set of minimizers of the thermotropic bulk energy density, $f_b$, in the $(q_1,q_2,q_3)$-plane can be written as $\mathbb{S}  = \mathbb{S}_1\cup\mathbb{S}_2$, where
\begin{gather}
\mathbb{S}_1 = \left\{(q_1,q_2,q_3):q_1^2+q_2^2=\frac{s_+^2}{4},\,\,q_3 = -\frac{s_+}{6}\right\},\qquad \mathbb{S}_2 = \left(0,0,\frac{s_+}{3}\right).
\end{gather}
The $\lambda\to\infty$ limit is equivalent to the vanishing elastic constant limit, and the bulk energy density, $f_b$, converges uniformly to its minimum value in this limit \cite{majumdar2010landau}.
\begin{proposition}\label{Oseen_frank_limit}
Let $\Omega\in \mathbb{R}^2$ be a simply connected bounded open set with smooth boundary. Let $(q_1^{\lambda},q_2^{\lambda},q_3^{\lambda})$ be a global minimizer of $J(q_1,q_2,q_3)$ in the admissible class $\mathcal{A}_0$ in \eqref{Ao}, when $L_2>-1$. Then there exists a sequence $\lambda_k\to\infty$ such that $(q_1^{\lambda_k},q_2^{\lambda_k},q_3^{\lambda_k})\to(q_1^{\infty},q_2^{\infty},q_3^{\infty})$ strongly in $W^{1,2}(\Omega;\mathbb{R}^3)$ where
$(q_1^{\infty},q_2^{\infty},q_3^{\infty})\in\mathbb{S}$.
If $(q_1^{\infty},q_2^{\infty},q_3^{\infty})\in \mathbb{S}_1$, i.e.,
\begin{gather}
    q_1^{\infty} = \frac{s_+}{2}\cos(2\theta^{\infty}),\quad q_2^{\infty} =  \frac{s_+}{2}\sin(2\theta^{\infty}),\quad q_3^{\infty} = -\frac{s_+}{6},
\end{gather} 
then $\theta^{\infty}$ is a minimizer of 
\begin{equation}\label{theta}
\int_{\Omega}|\nabla \theta|^2\textrm{dA},
\end{equation}
in the admissible class
\begin{equation}
\mathcal{A}_{\theta}=\{\theta\in 
W^{1,2}(\Omega);\theta = \theta_b\ on\ \partial\Omega\},
\end{equation}
The boundary condition, $\theta_b$, is compatible with $(q_1,q_2)$ on $\pp\Omega$ by the relation $(q_{1b},q_{2b})= \frac{s_+}{2}(\cos(2\theta_b),\sin(2\theta_b))$.
Otherwise, $(q_1^{\infty},q_2^{\infty},q_3^{\infty})(x,y)\in \mathbb{S}_2$, i.e.,
 \begin{gather}
(q_1^{\infty},q_2^{\infty},q_3^{\infty}) = \left(0,0,\frac{s_+}{3}\right).
\end{gather}
\end{proposition}
\begin{proof}
Our proof is analogous to Lemma 3 of \cite{majumdar2010landau}.
Firstly, we note that $(q_1^{\infty},q_2^{\infty},q_3^{\infty})$ belongs to admissible space $\mathcal{A}_0$.
As in Proposition \ref{prop2}, we can show that the $W^{1,2}$-norms of the $(q_1^{\lambda},q_2^{\lambda},q_3^{\lambda})$'s are uniformly bounded. Hence there exists a weakly-convergent subsequence $(q_1^{\lambda_k},q_2^{\lambda_k},q_3^{\lambda_k})$ such that $(q_1^{\lambda_k},q_2^{\lambda_k},q_3^{\lambda_k})\rightharpoonup(q_1^{1},q_2^{1},q_3^{1})$ in $W^{1,2}$, for some $(q_1^{1},q_2^{1},q_3^{1})\in\mathcal{A}_0$ as $\lambda_k\to\infty$. Using the lower semicontinuity of the $W^{1,2}$ norm with respect to the weak convergence, we have that 
\begin{equation}\label{1_leq_infty}
\int_{\Omega}|\nabla(q_1^{1},q_2^{1},q_3^{1})|^2\,\textrm{dA}\leq\int_{\Omega}|\nabla(q_1^{\infty},q_2^{\infty},q_3^{\infty})|^2\,\textrm{dA}
\end{equation}
The relation in \eqref{qinfty} shows that $\int_{\Omega}\bar{f}_b(q_1^{\lambda_k},q_2^{\lambda_k},q_3^{\lambda_k})\,\textrm{dA}\leq\frac{1}{\lambda_k^2}\int_{\Omega}f_{el}(q_1^{\infty},q_2^{\infty},q_3^{\infty})\,\textrm{dA}$, where $\bar{f}_b=\frac{1}{L}(f_b-\min f_b)$, and hence
\begin{gather}
   \int_{\Omega}\bar{f}_b(q_1^{\lambda_k},q_2^{\lambda_k},q_3^{\lambda_k})\,\textrm{dA}\to0\quad as \quad \lambda_k\to\infty. 
\end{gather}
Since $\bar{f}_b(q_1,q_2,q_3)\geq 0$, $\forall (q_1,q_2,q_3)\in\mathbb{R}^3$ we have that, on a subsequence $\lambda_{k_j}$, 
\begin{gather}
  \bar{f}_b(q_1^{\lambda_{k_j}}(\mathbf{x}),q_2^{\lambda_{k_j}}(\mathbf{x}),q_3^{\lambda_{k_j}}(\mathbf{x}))\to0,
\end{gather}
for almost all $\mathbf{x}\in\Omega$.
We know that $\bar{f}_b(q_1,q_2,q_3)=0$ if, and only if, $(q_1,q_2,q_3)\in\mathbb{S}$. On the other hand, the sequence $(q_1^{\lambda_k},q_2^{\lambda_k},q_3^{\lambda_k})$ converges weakly in $W^{1,2}$ and, on a subsequence, strongly in $L^2$ to $(q_1^{1},q_2^{1},q_3^{1})$. Therefore, the weak limit $(q_1^{1}(x),q_2^{1}(x),q_3^{1}(x))$ is in the set $\mathbb{S}$ a.e. $\mathbf{x}\in\Omega$.
If $(q_1^{1}(x),q_2^{1}(x),q_3^{1}(x))$ is in the set $\mathbb{S}_1$, then $|\nabla(q_1^1,q_2^1,q_3^1)|^2 = s_+^2|\nabla\theta^1|^2$, where $q_1^1 = \frac{s_+}{2}\cos(2\theta^1)$ and $q_2^1 = \frac{s_+}{2}\sin(2\theta^1)$. In this case, we take  $(q_1^{\infty},q_2^{\infty},q_3^{\infty})\in\mathbb{S}_1$ as defined above, so that $|\nabla(q_1^{\infty},q_2^{\infty},q_3^{\infty})|^2 = s_+^2|\nabla\theta^{\infty}|^2$, where $q_1^{\infty} =  \frac{s_+}{2}\cos(2\theta^{\infty})$ and $q_2^{\infty} =  \frac{s_+}{2}\sin(2\theta^{\infty})$.
If $(q_1^{1}(x),q_2^{1}(x),q_3^{1}(x))$ is in the set $\mathbb{S}_2$, then we take $(q_1^{\infty},q_2^{\infty},q_3^{\infty}) = (q_1^{1},q_2^{1},q_3^{1}) = (0,0,s_+/3)$. Combining \eqref{1_leq_infty} with the definition of $(q_1^{\infty},q_2^{\infty},q_2^{\infty})$, we obtain that $\int_{\Omega}|\nabla\theta^1|^2\,\textrm{dA}=\int_{\Omega}|\nabla\theta^{\infty}|^2\,\textrm{dA}$ and $\int_{\Omega}|\nabla(q_1^1,q_2^1,q_3^1)|^2\,\textrm{dA}=\int_{\Omega}|\nabla(q_1^{\infty},q_2^{\infty},q_3^{\infty})|^2\,\textrm{dA}$.
Therefore,
\begin{align}
\int_{\Omega}|\nabla(q_1^{\infty}, q_2^{\infty}, q_3^{\infty})|^2\,\textrm{dA}&\leq\liminf_{\lambda_{k_j}\to\infty}\int_{\Omega}|\nabla(q_1^{\lambda_{k_j}},q_2^{\lambda_{k_j}},q_3^{\lambda_{k_j}})|^2\,\textrm{dA}\\
&\leq\limsup_{\lambda_{k_j}\to\infty}\int_{\Omega}|\nabla(q_1^{\lambda_{k_j}},q_2^{\lambda_{k_j}},q_3^{\lambda_{k_j}})|^2\,\textrm{dA}\\
&\leq\int_{\Omega}|\nabla(q_1^{\infty},q_2^{\infty},q_3^{\infty})|^2\,\textrm{dA},
\end{align}
which demonstrates that $\lim\limits_{\lambda_{k_j}\to\infty}||\nabla(q_1^{\lambda_{k_j}},q_2^{\lambda_{k_j}},q_3^{\lambda_{k_j}})||_{L^2}=||\nabla(q_1^{\infty},q_2^{\infty},q_3^{\infty})||_{L^2}$. This together with the weak convergence $(q_1^{\lambda_{k_j}},q_2^{\lambda_{k_j}},q_3^{\lambda_{k_j}})\to(q_1^{\infty},q_2^{\infty},q_3^{\infty})$ suffices to show the strong convergence $(q_1^{\lambda_{k_j}},q_2^{\lambda_{k_j}},q_3^{\lambda_{k_j}})\to(q_1^{\infty},q_2^{\infty},q_3^{\infty})$ in $W^{1,2}$. Since $f_b(q_1^{\infty},q_2^{\infty},q_3^{\infty})$ is constant for $(q_1^{\infty},q_2^{\infty},q_3^{\infty})\in\mathbb{S}_1$, we can recast our minimization problem to minimizing the elastic energy alone i.e., minimization of the limiting functional:
\begin{equation}\label{eq:J_infty}
J_{\infty} [q_1, q_2, q_3]= \int_{\Omega}f_{el}\left(q_1, q_2, q_3 \right)\,\mathrm{dA}.
\end{equation}
Substituting
\begin{gather}
q_1 = \frac{s_+}{2}\cos(2\theta),\quad
q_2 = \frac{s_+}{2}\sin(2\theta),\quad
q_3 = -\frac{s_+}{6},
\end{gather}
into \eqref{eq:J_infty}, we have that
\begin{align}
J_{\infty} = s_+^2\left(1+\frac{L_2}{2}\right)\int_{\Omega}|\nabla\theta|^2\,\mathrm{dA}. \label{upperenergybound}
\end{align}
The corresponding Euler-Lagrange equation for \eqref{theta} is simply the Laplace equation
\begin{equation}
\Delta \theta = 0,
\end{equation}
subject to $\theta = \theta_b$ on $\partial \Omega$, where $(q_1,q_2,q_3) = (\frac{s_+}{2}\cos(2\theta_b),\frac{s_+}{2}\sin(2\theta_b),-\frac{s_+}{6})$ on $\pp\Omega$. 
\end{proof}

For large $\lambda$ i.e., large square domains, there are two classes of stable equilibria which are almost in the set $\mathbb{S}_1$. The diagonal states, $D$, are such that the nematic director (in the plane) is aligned along one of the square diagonals. The rotated states, labelled as $R$ states, are such that the director rotates by $\pi$ radians between a pair of opposite square edges. There are $2$ rotationally equivalent $D$ states, and $4$ rotationally equivalent $R$ states. The corresponding boundary conditions in terms of $\theta$
are given by $\theta_b = \theta_b^D$ or $\theta_b^R$ respectively, where 
\begin{gather}
\begin{cases}
\theta_b^D &= \frac{\pi}{2},\ \text{on}\ x = \pm1,\\
\theta_b^D &= 0,\ \text{on}\ y = \pm1,
\end{cases}
\qquad
\begin{cases}
\theta_b^R &= \frac{\pi}{2},\ \text{on}\ x = -1,\\
\theta_b^R &= -\frac{\pi}{2},\ \text{on}\ x = 1,\\
\theta_b^R &= 0,\ \text{on}\ y = \pm1.
\end{cases}
\end{gather} 
\begin{figure}[htbp]
    \begin{center}
    \includegraphics[width=0.7\columnwidth]{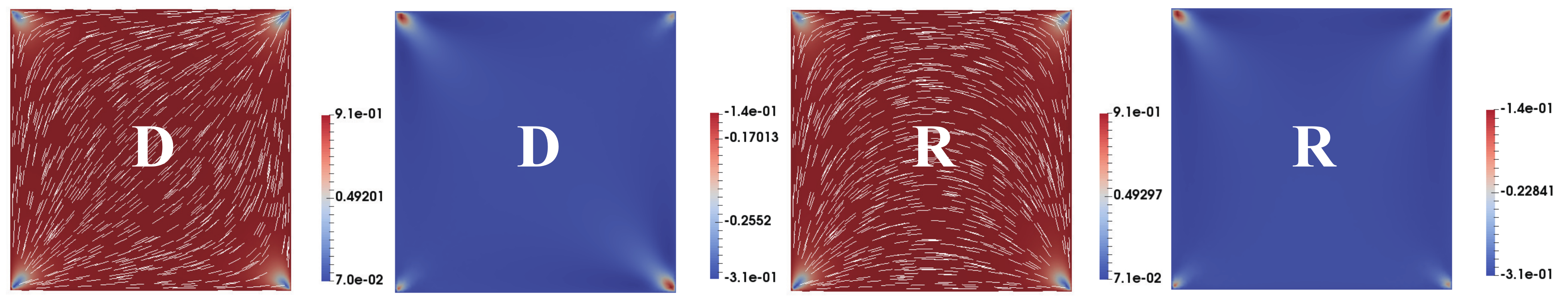}
    \caption{The $D$ and $R$ solutions of the Euler-Lagrange equation \eqref{q1eq}--\eqref{q3eq} with $L_2 = 3.5$ and $\bar{\lambda}^2= 1000$. In the first and third columns, we plot the order parameter, $s^2$, and the director profile, $\mathbf{n}$. The second and fourth columns are the corresponding plots of $q_3$. We see that $q_3$ is approximately constant on $\Omega$, except for near the vertices.}
    \label{fig:D_R}
    \end{center}
\end{figure}
\begin{figure}[htbp]
    \begin{center}
    \includegraphics[width=0.9\columnwidth]{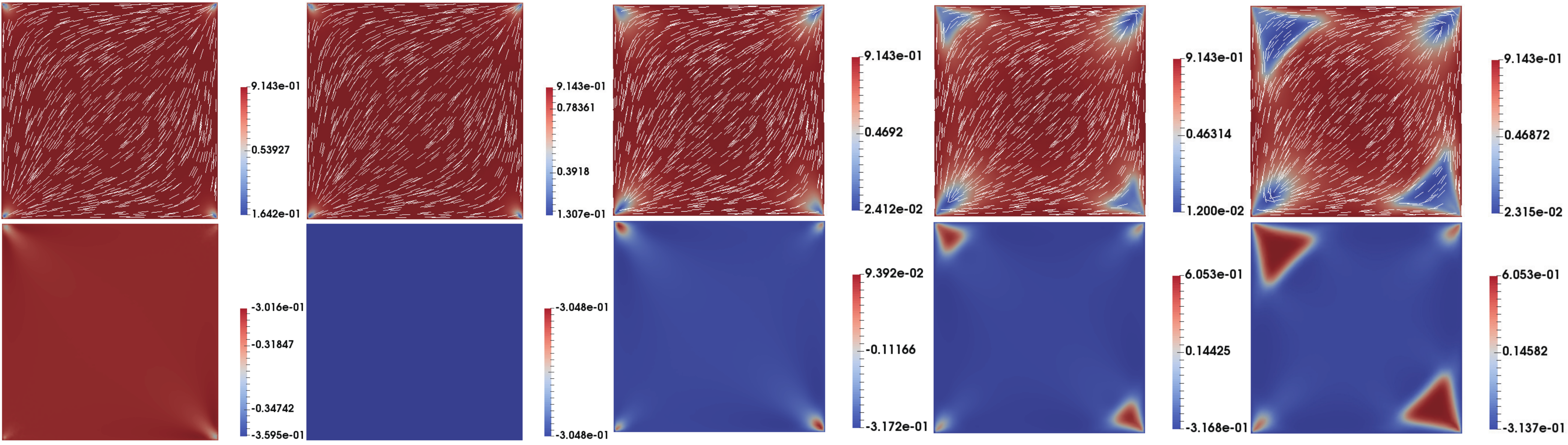}
    \caption{The $D$ solution of the equations, \eqref{q1eq}--\eqref{q3eq}, with $\bar{\lambda}^2= 1000$, with $L_2 = -0.5, 0, 10,30$ and $45$, respectively. In the first row, we plot the order parameter, $s^2$, and the director profile $\mathbf{n}$. In the second row, we plot the corresponding profiles of $q_3$. The effects of increasing $L_2$ are most pronounced near the splay vertices.}
    \label{fig:D}
    \end{center}   
\end{figure}

In Figure \ref{fig:D_R}, we plot a $D$ and $R$ solution with $L_2 = 3.5$ and $\bar{\lambda}^2 = 1000$. In Figure \ref{fig:D}, we study the effect of increasing $L_2$ on a $D$ state with $\bar{\lambda}^2 = 1000$. When $L_2 = 0$, we see that $s^2 = q_1^2+q_2^2\approx s_+^2/4$, $q_3= -s_+/6$ almost everywhere on $\Omega$. In \cite{han2020pol}, the authors show that the limiting profiles described in Proposition~\ref{Oseen_frank_limit} are a good approximation to the solutions of \eqref{q1eq}-\eqref{q3eq}, for large $\lambda$.
 The differences between the limiting profiles and the numerically computed $D$ solutions concentrate around the vertices, for large $\lambda$. A square vertex is referred to as \emph{splay} or \emph{bend} vertex, according to whether the planar director rotates by $\pi/2$ or $-\pi/2$ radians along a circle centered at the vertex, oriented in an anticlockwise sense. As $|L_2|$ increases, $q_3$ deviates significantly from the limiting value $q_3^\infty= -s_+/6$, near the square vertices; the deviation being more significant near the bend vertices compared to the splay vertices. Notably, the value of $q_3$ near the vertices increases as $L_2$ increases and, from an optical perspective, we expect to observe larger defects near the square vertices for more anisotropic materials with $L_2 \gg 1$, on large square domains. 

In \cite{lewis2015defects}, the authors compute the limiting energy, $J_\infty$, of $D$ and $R$ solutions on a unit square domain $\Omega$ to be:
\begin{align}
J_{\infty}(D) &= 2\pi s_+^2\left(1+\frac{L_2}{2}\right)\left(ln\left(\frac{1}{\epsilon}\right)+ln\left(\frac{2}{\pi}\right)+s_1-s_2+O(\epsilon^2)\right),\\
J_{\infty}(R) &= 2\pi s_+^2\left(1+\frac{L_2}{2}\right)\left(ln\left(\frac{1}{\epsilon}\right)+ln\left(\frac{2}{\pi}\right)+s_1+s_2+O(\epsilon^2)\right),
\end{align}
respectively where 
\begin{gather*}
s_1 = 2\sum_{n=0}^{\infty}\frac{\coth((2n+1)\pi)-1}{2n+1},\qquad
s_2 = 2\sum_{n=0}^{\infty}\frac{\mathrm{csch}((2n+1)\pi)}{2n+1}.
\end{gather*}
Since $\mathrm{csch}((2n+1)\pi)$ is positive, we have $J_{\infty}(D)<J_{\infty}(R)$.
The numerical values of $ln(2/\pi)+s_1-s_2$ and $ln(2/\pi)+s_1+s_2$ are approximately zero, so $J_{\infty}(D)$ and $J_{\infty}(R)$ are approximately $ln(1/\epsilon)$ for small $\epsilon$, and the limiting energies are linear in $L_2$.

The $Constant$ solution in $\mathbb{S}_2$ has transition layers on the boundary from $(0,0,s_+/3)$ to $(s_+/2,0,-s_+/6)$ or $(- s_+/2,0,-s_+/6)$.
Analogous to Section 4 of \cite{wang2019order}, we define a metric $d$ on the $(q_1, q_2, q_3)$-plane in the following way: for any two points $\mathbf{q}_0,\mathbf{q}_1\in\mathbb{R}^3$,
\begin{equation}\label{geodesic}
d(\mathbf{q}_0,\mathbf{q}_1) = \inf\left\{\int_0^1 F^{1/2}(\mathbf{q}(t))|\mathbf{q}'(t)|dt:\mathbf{q}\in C^1([0,1];\mathbb{R}^3), \mathbf{q}(0) = \mathbf{q}_0, \mathbf{q}(1) = \mathbf{q}_1\right\},
\end{equation}
which is the geodesic distance associated with the Riemannian metric $F^{1/2}$, where $F = f_b-\min f_b$. This metric is degenerate in the sense that $F(\mathbf{q}) = F(0,0,s_+/3) = 0$, $\mathbf{q}\in \mathbb{S}_1$. As in the arguments in Lemma 9 of \cite{1988The}, the infimum in \eqref{geodesic} is achieved by a minimizing geodesic for any $\mathbf{q}_0$ and $\mathbf{q_1}$. There exists a subsequence $\lambda_j\nearrow+\infty$ such that $\mathbf{q}_{\lambda_j}$ converges in $L^1(\Omega)$ and a.e., to a map of the form $\mathbf{q}_{\infty} = \mathbf{q}$, where $\mathbf{q}\in \mathbb{S}$ is a minimizer of the functional:
\begin{equation}\label{G_infty}
G_{\infty} = \int_{\partial\Omega}d(\mathbf{q}_{\infty}(\mathbf{r}),\mathbf{q}_b(\mathbf{r}))\,\mathrm{d}\mathcal{H}^1(\mathbf{r}),
\end{equation}
where 
\begin{equation}
\mathcal{H}^1(\partial\Omega) = \sup\left\{\int_{\Omega} \Div u\,\mathrm{dA}: u\in C_c^1(\Omega),|u|\leq 1\ \text{on}\ \Omega\right\}
\end{equation}
denotes the length of $\partial\Omega$ and formally is its 1-dimensional Hausdorff measure.
Let us introduce the transition cost
\begin{gather}
c_1 := d(\mathbf{q},(0,0,\frac{s_+}{3})),\quad \mathbf{q}\in\mathbb{S}_1.
\end{gather}
Using \eqref{G_infty}, the energy of this $Constant$ configuration, as $\lambda\to+\infty$, is given by $G_{\infty}(Constant) = 4c_1$.
The numerical value of $c_1$ is $41.6817$ in \cite{wang2019order}. $G_{\infty}(Constant)$ is independent of $L_2$. Hence, there is a critical value 
\begin{gather*}
    L_2^* = \frac{4c_1}{s_+^2\pi(ln(\frac{1}{\epsilon})+ln(\frac{2}{\pi})+s_1-s_2+O(\epsilon^2))}-2,
\end{gather*}
such that for $L_2>L_2^*$, the limiting $Constant$ solution is energetically preferable to the $D$ and $R$ solutions i.e., $G_{\infty}(Constant)<J_{\infty}(D)<J_{\infty}(R)$.

\subsection{The Novel pWORS Solutions}
 For all $\lambda > 0$ and $L_2 = 0$, the $WORS$ is a solution of \eqref{q1eq}--\eqref{q3eq} given by $(q,0,-B/6C)$, where $q$ satisfies \eqref{WORS}. In Section \ref{sec:smallL2smalllambda}, we study the Euler-Lagrange equations, in the small $\lambda$ and small $L_2$ limit, up to $\mathcal{O}(L_2)$; see (\ref{feq})--(\ref{heq}). However, $g\equiv 0$ is a solution of \eqref{geq} for all $\lambda$, and so we need to consider terms of $\mathcal{O}(L_2^2)$ when dealing with the $q_2$ component. We  assume that the solution, $(q_1, q_2, q_3)$, of \eqref{q1eq}-\eqref{q3eq}, can be expanded as follows:
\begin{align*}
q_1(x,y)&=q(x,y)+L_2f(x,y)+L_2^2\varphi(x,y)+\dots\\   
q_2(x,y)&=0+L_2g(x,y)+L_2^2\gamma(x,y)+\dots\\
q_3(x,y)&=-\frac{B}{6C}+L_2h(x,y)+L_2^2\mu(x,y)+\dots 
\end{align*}
Using the $\mathcal{O}(L_2)$ equations in (\ref{feq})--(\ref{heq}) with respect to the quadruples $(q,f,g,h)$ and rearranging, we can calculate the second order Euler-Lagrange system given by the corresponding partial differential equations for $\varphi, \gamma, \mu$ as shown below:
\begin{align}
\Delta \varphi+\frac{1}{2}(h_{yy}-h_{xx})=&\frac{2C\lambda^2}{L}\left\{q^2(2\varphi-f)+q(3f^2+g^2+3h^2)-\frac{1}{4}(2f-q-4\varphi)\left(q^2-\frac{B^2}{4C^2}\right)\right\}, \\
\Delta \gamma-h_{xy}=&\frac{2C\lambda^2}{L}\left\{2qfg+\left(\gamma-\frac{1}{2}g\right)\left(q^2-\frac{B^2}{4C^2}\right)\right\},\label{g2eq}\\
\Delta \mu+\frac{1}{6}(f_{yy}-f_{xx})-\frac{1}{3}g_{xy}=&\frac{2C\lambda^2}{L}\left\{2h\left(qf-\frac{B}{C}h\right)+\left(\mu-\frac{1}{6}h\right)\left(q^2+\frac{B^2}{4C^2}\right)\right\}+\frac{1}{36}(q_{yy}-q_{xx}),
\end{align}
where $\varphi=\gamma=\mu=0$ on $\pp\Omega$. In Fig. \ref{fig:qfgh}, we plot a branch of the $\gamma$ solutions of \eqref{g2eq}. As $\lambda$ increases, we observe an increasing number of zeroes on the square diagonals, where $\gamma = 0$.
\begin{figure}[htb!]
    \begin{center}
     \includegraphics[width=0.85\columnwidth]{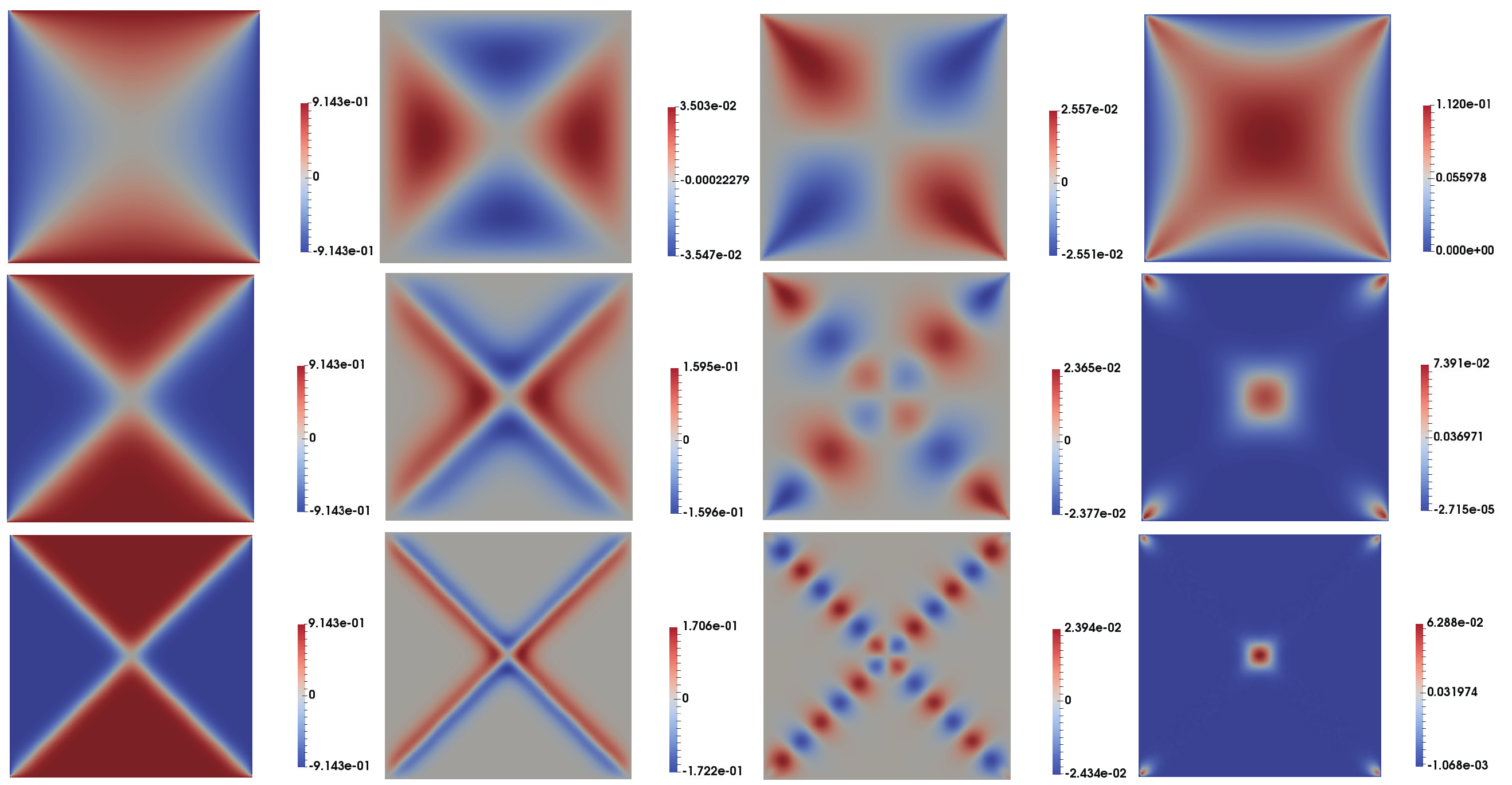}
         \caption{Plots of $q,f,\gamma$ and $h$ i.e., the solutions of (from left to right) \eqref{WORS},\eqref{feq},\eqref{g2eq} and \eqref{heq}, respectively with $\bar{\lambda}^2 = 5$, $100$ and $500$ from the first to the third row, respectively.}
          \label{fig:qfgh}
    \end{center}
\end{figure}
For any $\lambda>0$, we can use the initial condition $(q_1,q_2,q_3) = (q+L_2f,L_2g+L_2^2\gamma,-\frac{s_+}{6}+L_2h)$ to numerically find a new branch of unstable solutions, referred to as $pWORS$ configurations in Figure \ref{fig:pWORS}. $f,g,h,\gamma$ are the solutions of \eqref{feq}, \eqref{geq}, \eqref{heq}, and \eqref{g2eq}, respectively. In the $(q_1,q_2)$ plane, the $pWORS$ has a constant set of eigenvectors away from the diagonals, and has multiple $\pm 1/2$-point defects on the two diagonals, so that the $pWORS$ is similar to the $WORS$ away from the square diagonals. As $\lambda$ increases, the number of alternating $+1/2$ and $-1/2$ point defects on the square diagonals increases, for the numerically computed $pWORS$. This is mirrored by the function $\gamma$ in \eqref{g2eq} that encodes the second order effect of $L_2$ on the $WORS$.
\begin{figure}[htbp]
    \begin{center}
    \includegraphics[width=0.85\columnwidth]{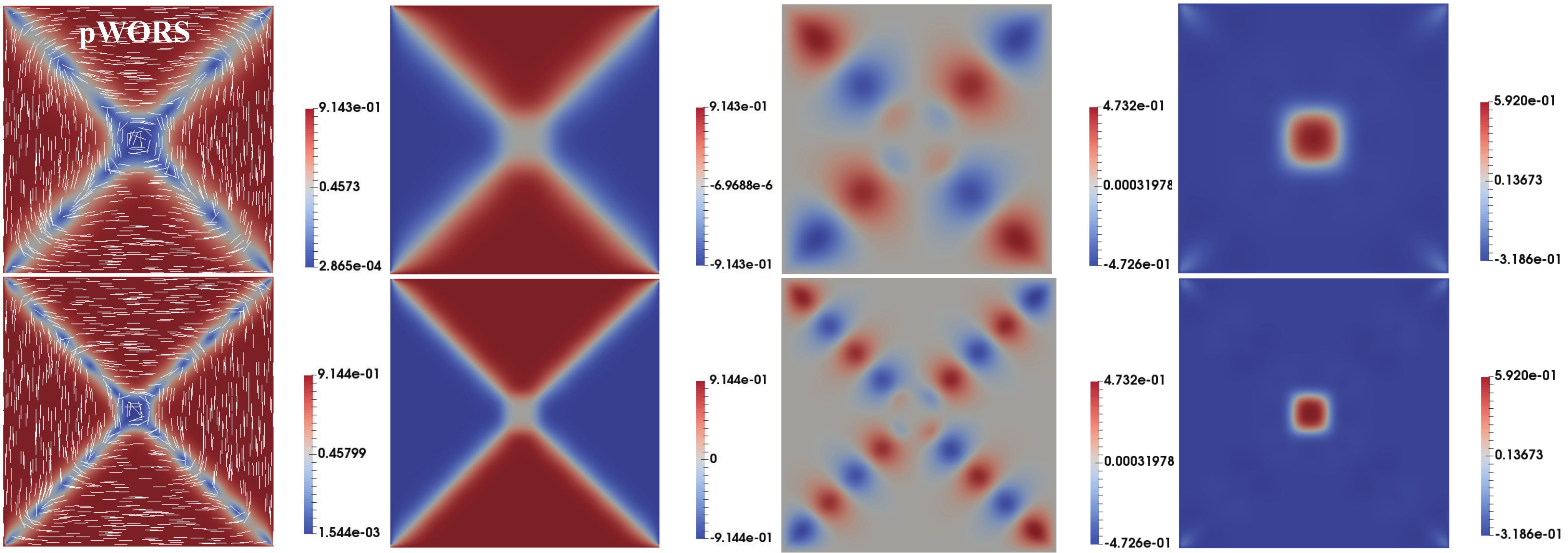}
    \caption{The configuration of the numerically computed $pWORS$. In the first column, we plot the order parameter, $s^2$, and the director profile $\mathbf{n}$. The second to fourth columns are the plots of the solutions of  the Euler-Lagrange equations \eqref{q1eq}--\eqref{q3eq}, with $L_2 = 3.5$ and $\bar{\lambda}^2 = 350,1000$ in the first and second rows, respectively.}
    \label{fig:pWORS}
    \end{center}
\end{figure}

\section{Bifurcation diagrams}
\label{sec:bifurcations}
We use the open-source package FEniCS \cite{logg2012automated} to perform all the finite-element simulations, numerical integration, and stability checks in this paper \cite{AlnaesBlechta2015a, LoggMardalEtAl2012a}. We apply the finite element method on a triangular mesh with mesh-size $h\leq 1/256$, for the discretization of a square domain. The non-linear equations \eqref{q1eq}-\eqref{q3eq} are solved by Newton's methods with a linear LU solver at each iteration. The tolerance for convergence is set to 1e-13. We check the stability of the numerical solution by numerically calculating the smallest real eigenvalue of the Hessian matrix of the energy functional \eqref{funcq123}, using the LOBPCG (locally optimal block preconditioned conjugate gradient) method \cite{knyazev2001toward}. If the smallest real eigenvalue is negative, the solution is unstable, and stable otherwise. In what follows, we compute bifurcation diagrams for the solution landscapes, as a function of $\bar{\lambda}^2$, with fixed temperature $A = -B^2/3C$, for five different values of $L_2 = 0, 1, 2.6, 3, 10$. The $C$ and $L$ are fixed material-dependent constants, so $\bar{\lambda}$ is proportional to $\lambda$ and we will use these diagrams to infer qualitative solution trends in terms of the edge length, $\lambda$. 

For $\lambda$ small enough,
there is a unique solution for any value of $L_2$; see the results in Section~\ref{sec:asymptotics}. For $L_2=0$, the unique stable solution for small enough $\lambda$
is the $WORS$. The unique solution deforms to the $Ring^+$, with a central point defect, for $L_2 =1, 2.6$. For $L_2 = 3, 10$, the unique solution is the \emph{Constant} solution, on the grounds that this solution approaches the constant state, $(q_1, q_2, q_3) \to (0, 0, s_+/3)$, in the square interior as $\lambda \to \infty$. In Figure \ref{fig:lambda_100_200}, we plot the energies of the $WORS$, $Ring^+$, and $Constant$ solutions for two distinct values of $\bar{\lambda}^2$, as a function of $L_2$. The energy is taken to be $J[q_1,q_2,q_3]$ in \eqref{funcq123} minus $\int_{\Omega}\min f_b\, \mathrm{dA}$, where the value of $\min f_b$ is in \eqref{minfb}, so that the energy is non-negative by definition.
\begin{figure}[htb!]
    \centering        
    \includegraphics[width=0.7\columnwidth]{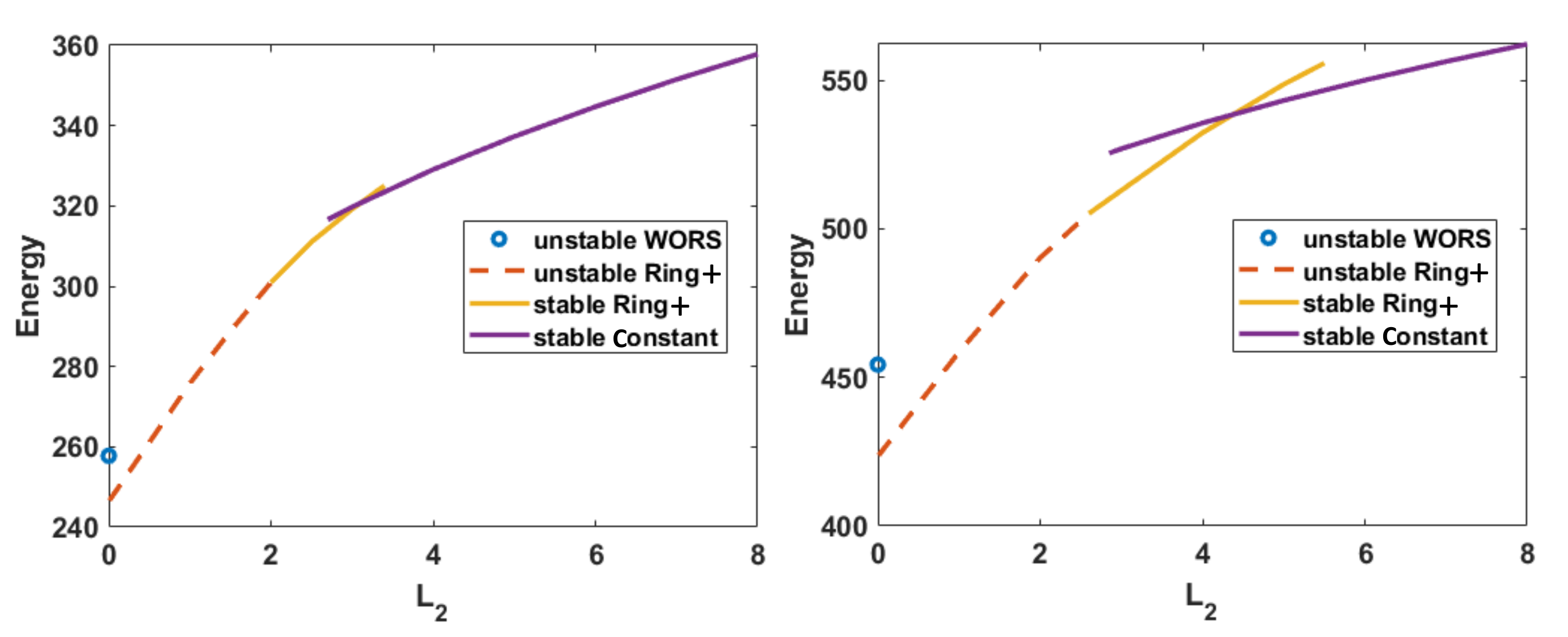}
    \caption{Energy plots, $J-\int_{\Omega}\min f_b\,dA$, verses $L_2$ when $\bar{\lambda}^2 = 100$ and $200$, respectively.}
    \label{fig:lambda_100_200}
\end{figure}

The $WORS$ only exists for $L_2 = 0$. The $Ring^+$ solution branch gains stability as $L_2$ increases. The $Ring^+$ and $Constant$ solution branches coexist for some values of $L_2$ ($2.7\leq L_2\leq 3.4$ for $\bar{\lambda}^2 = 100$, $2.85\leq L_2 \leq 5.5$ for $\bar{\lambda}^2 = 200$). When $L_2$ is large enough, the $Constant$ solution has lower energy than the $Ring^+$ solution. 
When $\lambda<\lambda_0$, there is unique solution for any $L_2$, which means 
the $WORS$, $Ring^+$, and $Constant$ solution branches are connected.

We distinguish between the distinct solution branches by defining two measures $\int_{\Omega} q_1(1+x+y)\mathrm{d}x\mathrm{d}y$ and $\int_{\Omega} q_2(1+x+y)\mathrm{d}x\mathrm{d}y$. In addition to the $WORS$, $Ring^+$, \emph{Constant} solutions, there also exist the unstable $Ring^-$ and unstable $pWORS$ solution branches with the same symmetries in Proposition \ref{prop:forever_critical}, which are indistinguishable by these measures. Hence, they appear on the same line in bifurcation diagram in Figure \ref{fig:bifurcation_diagram_10} for all $L_2>0$. The difference between the $Ring^+$, $Ring^-$, $WORS$, $Constant$, and $pWORS$ can be spotted from the associated $q_2$-profiles. If $q_2<0$ on $x=y$ and $x>0$, the corresponding solution is the $Ring^+$ solution. If $q_2>0$ on $x = y$ and $x>0$, the corresponding solution is the $Ring^-$ solution. The $Ring^+$ and $Ring^-$ solutions also exist for $L_2= 0$. If $q_2\equiv 0$, the solution is either the $WORS$ or $Constant$ solutions. If $q_2$ has isolated zero points on the square diagonals, the corresponding solution is identified to be the $pWORS$ solution branch.

We numerically solve the Euler-Lagrange equations \eqref{q1eq}-\eqref{q3eq} with $\bar{\lambda}^2 = 0.1$ by using Newton's method to obtain: the unique stable $WORS$ with $L_2 = 0$; the $Ring^+$ solution with $L_2 = 1,2.6$ and; the $Constant$ solution with $L_2 = 3,10$. The initial condition is not important here, since the solution is unique and the nonlinear term is small for $\bar{\lambda}^2 = 0.1$. We perform an increasing $\bar{\lambda}^2$ sweep for the $WORS$, $Ring^+$ and $Constant$ solution branches and a decreasing $\bar{\lambda}^2$ sweep for the diagonal $D$, and rotated $R$ solution branches (as described in the Section~\ref{sec:lambdainfty}). The stable $Ring^+$ branch for $L_2 = 3$ is obtained by taking the stable $Ring^+$ branch, with $L_2 = 2.6$ as the initial condition. The unstable $WORS$ and $Ring^+$ are tracked by continuing the stable $WORS$ and stable $Ring^+$ branches.
If the $Ring^+$ branch is given by $(q_1,q_2,q_3)$ for a fixed $L_2>0$, then the initial condition for the unstable $Ring^-$-solution is given by the corresponding $(q_1,-q_2,q_3)$ solution, for any $\lambda>0$. 
The initial condition for the unstable $pWORS$ branch is given by $(q_1,q_2,q_3) = (q+L_2f,L_2g+L_2^2\gamma,-\frac{s_+}{6}+L_2h)$, where $q,f,g,h,\gamma$ are the solutions of \eqref{feq}--\eqref{heq}, and \eqref{g2eq}, respectively, for any $\lambda>0$ (see Fig.~\ref{fig:pWORS}).


\begin{figure}
\centering
    \begin{subfigure}{0.55\textwidth}
        \centering
        \includegraphics[width=\columnwidth]{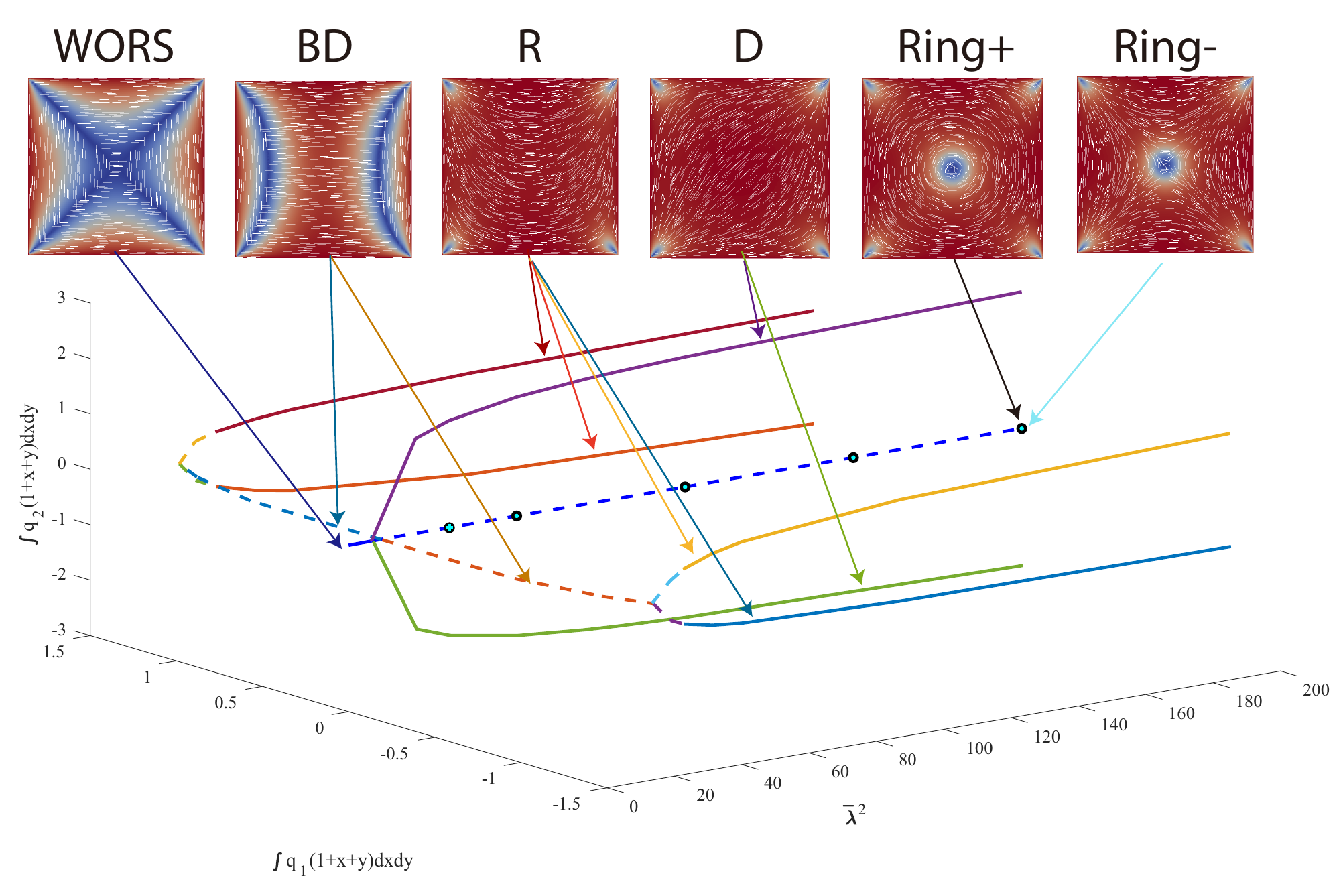}
    \end{subfigure}
        \vspace{2em}
    \begin{subfigure}{0.4\textwidth}
        \centering
        \includegraphics[width=\columnwidth]{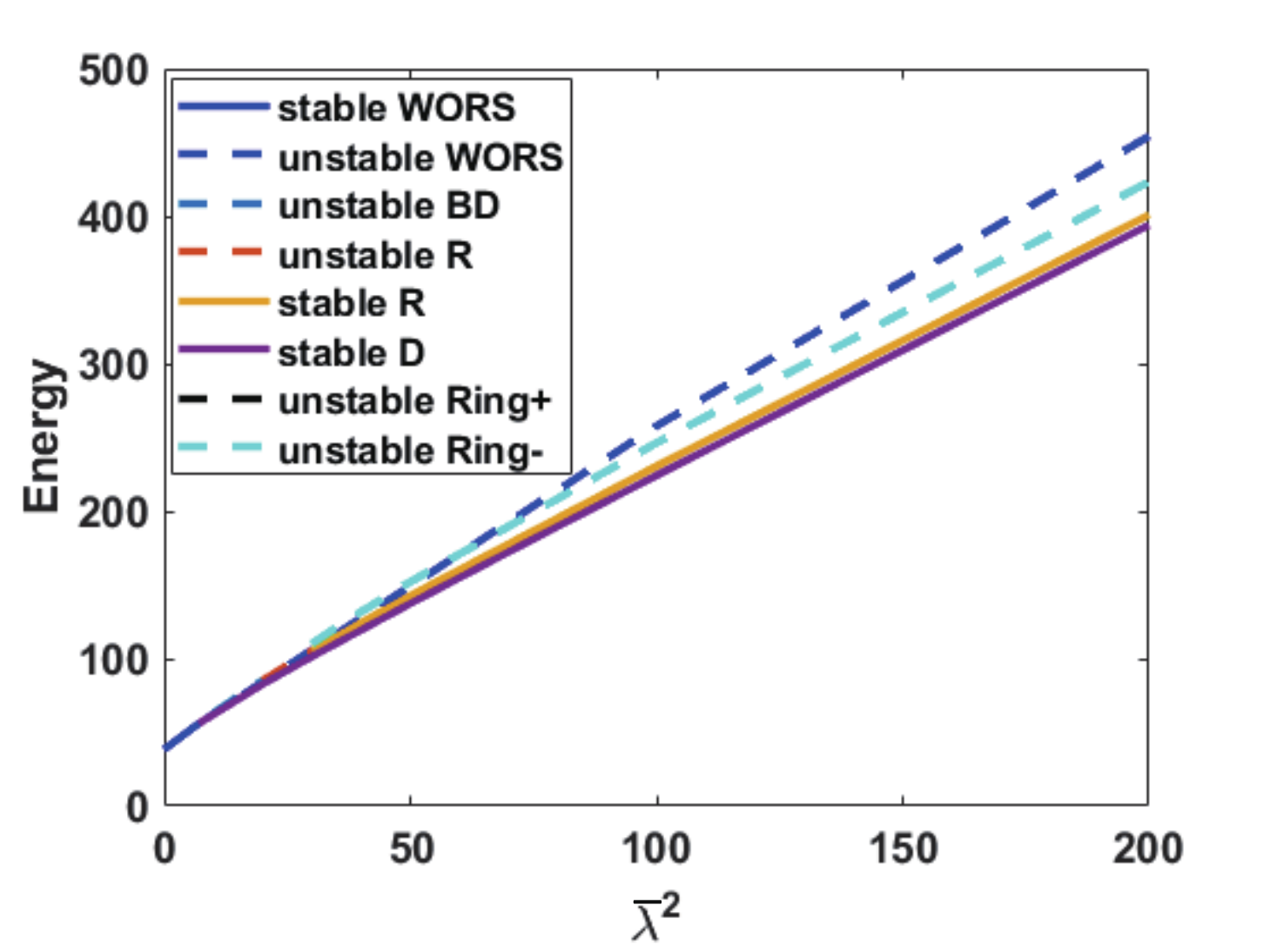}
    \end{subfigure}
    \begin{subfigure}{0.55\textwidth}
        \centering
        \includegraphics[width=\columnwidth]{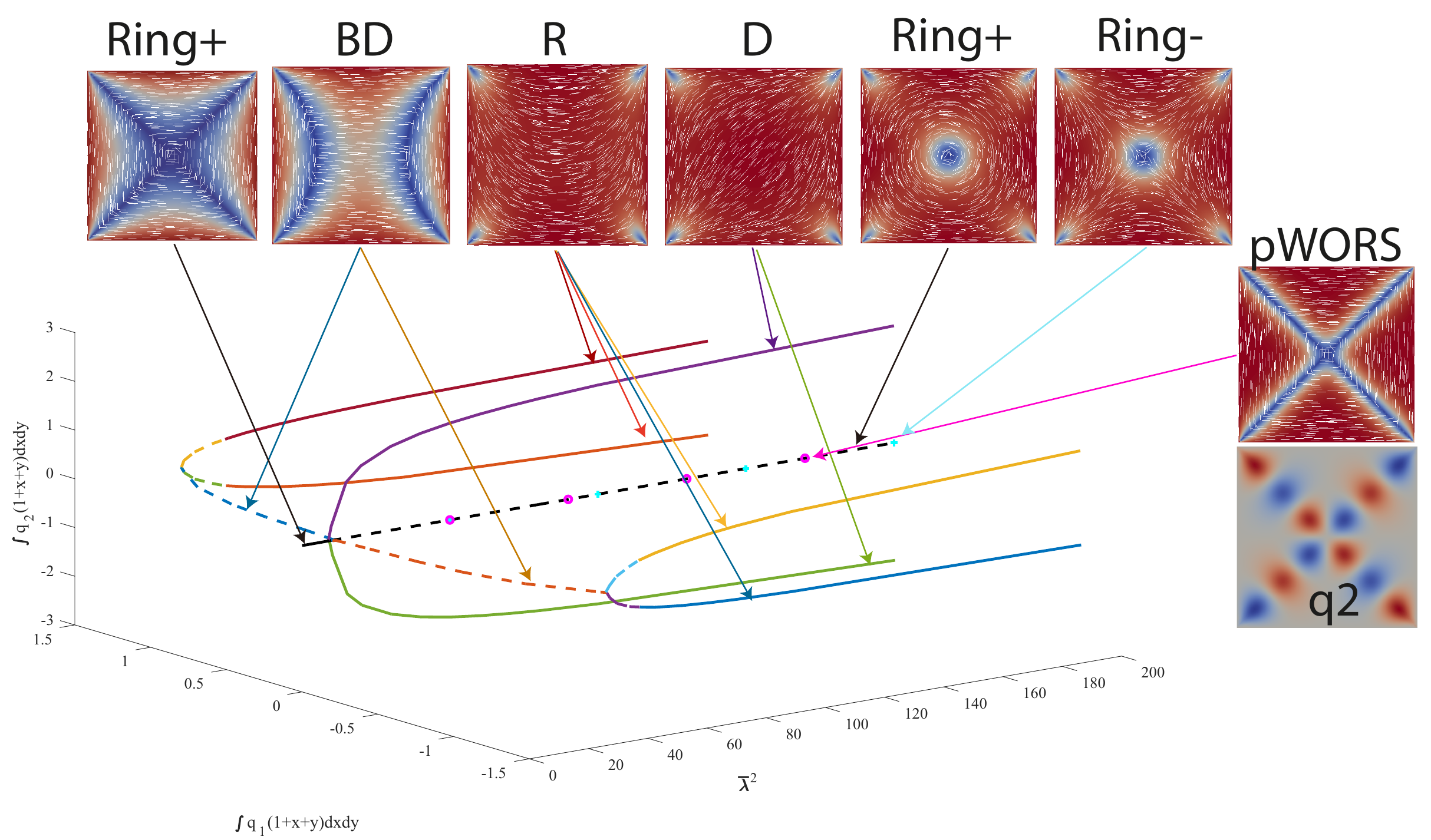}
    \end{subfigure}
    \vspace{2em}
    \begin{subfigure}{0.4\textwidth}
        \centering
        \includegraphics[width=\columnwidth]{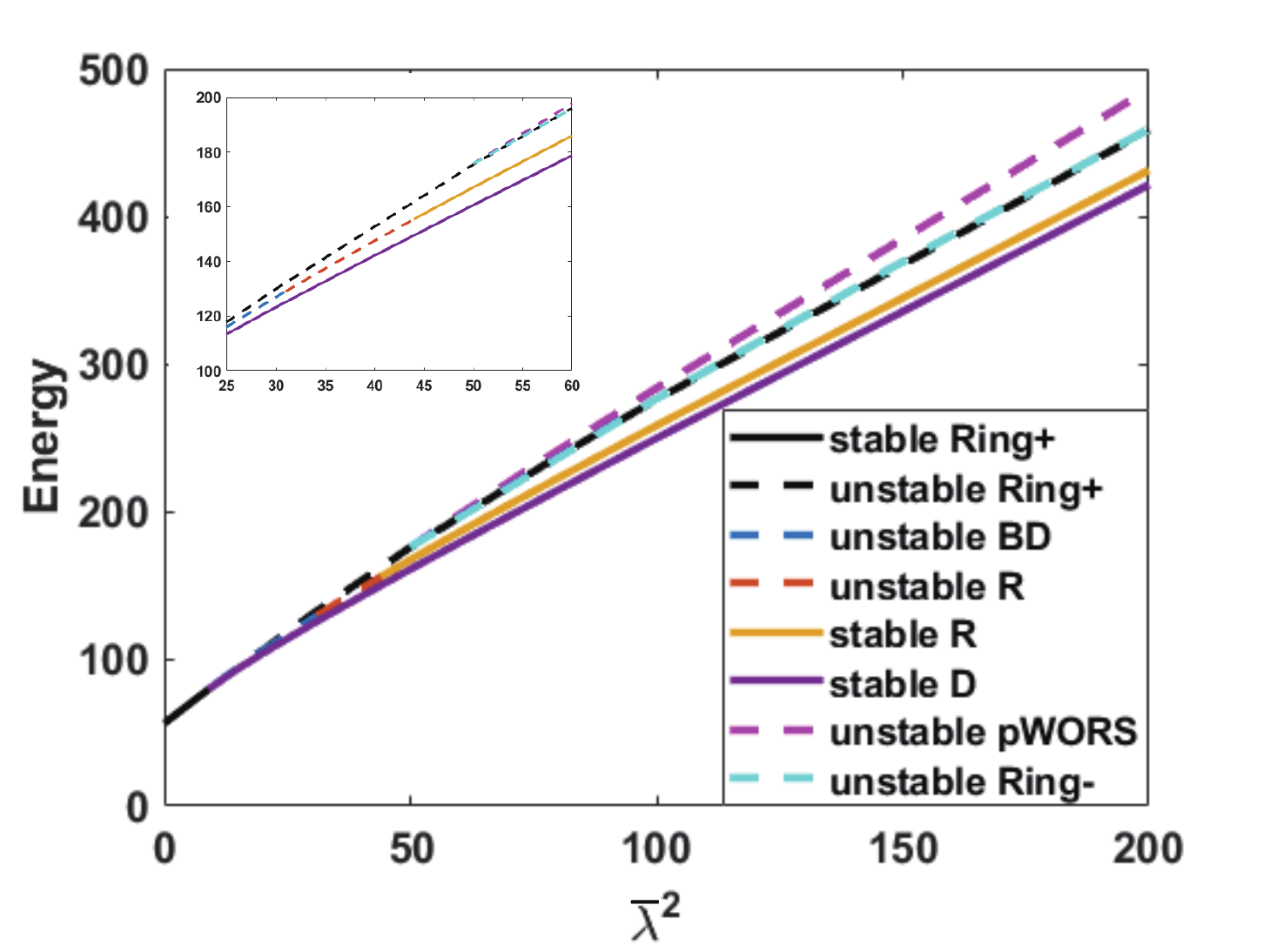}
    \end{subfigure}
    \begin{subfigure}{0.55\textwidth}
        \centering
        \includegraphics[width=\columnwidth]{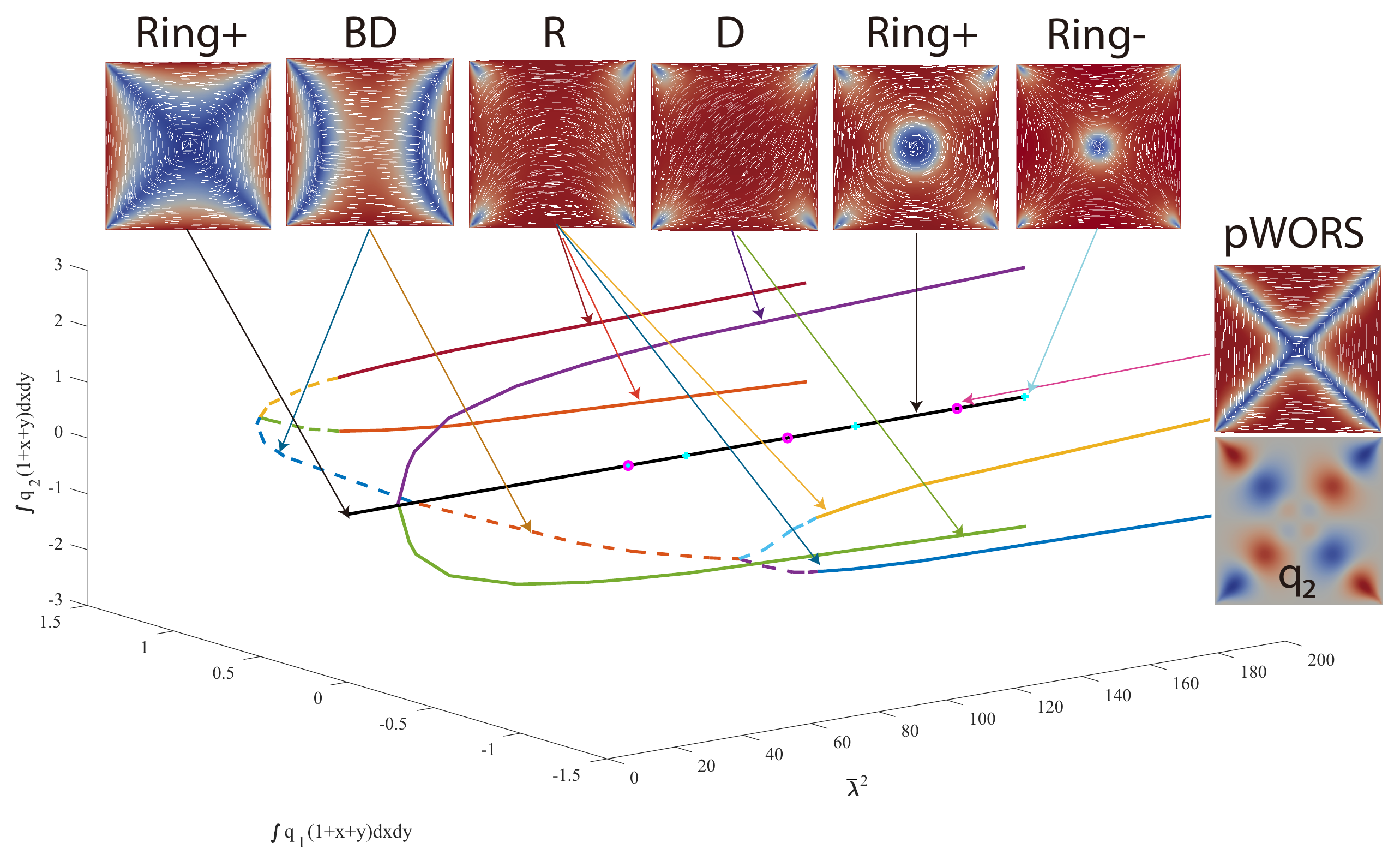}
    \end{subfigure}
    \begin{subfigure}{0.4\textwidth}
        \centering
        \includegraphics[width=\columnwidth]{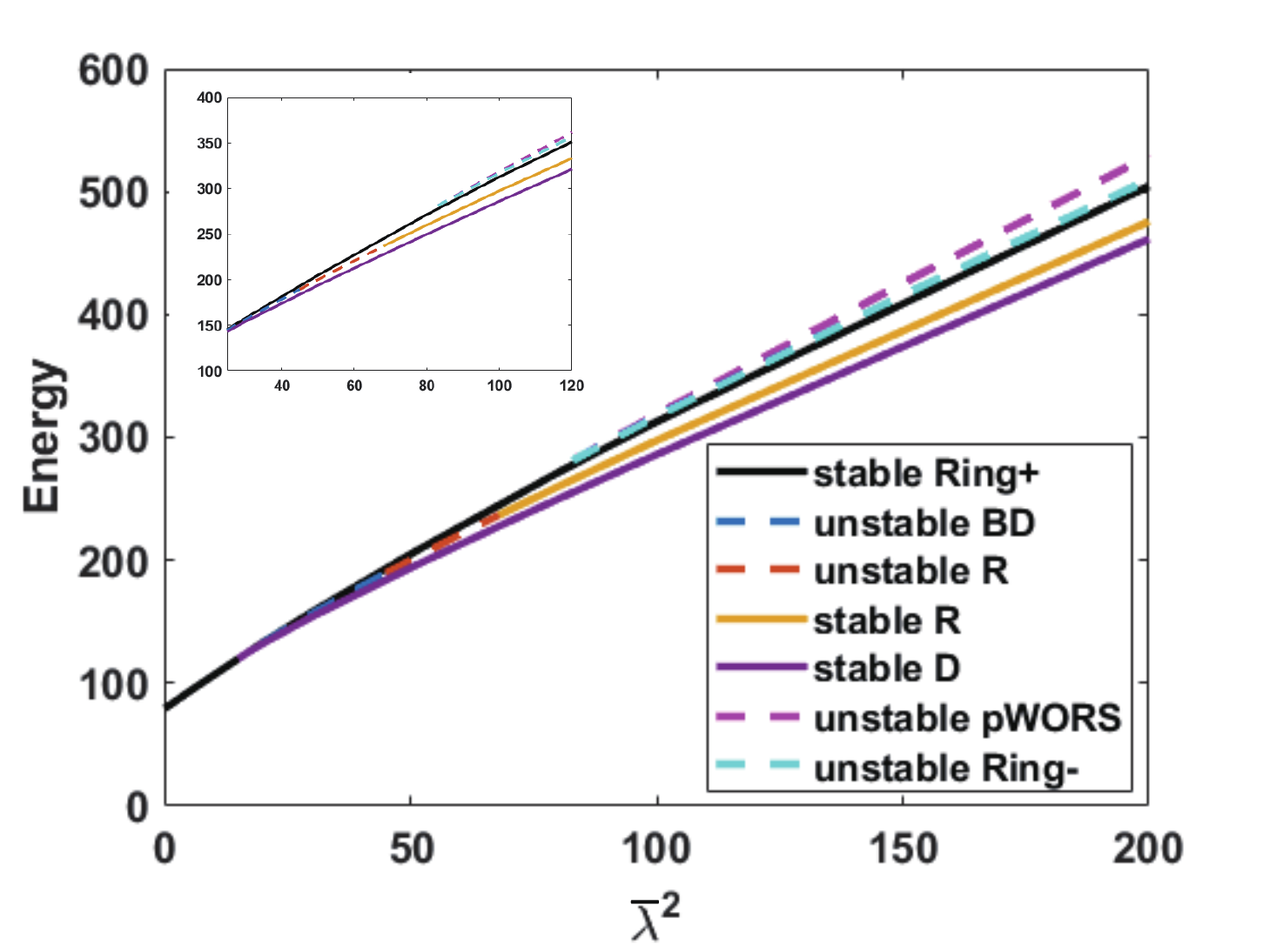}
    \end{subfigure}
\end{figure}
\begin{figure}
\centering
    \begin{subfigure}{0.55\textwidth}
        \centering
        \includegraphics[width=\columnwidth]{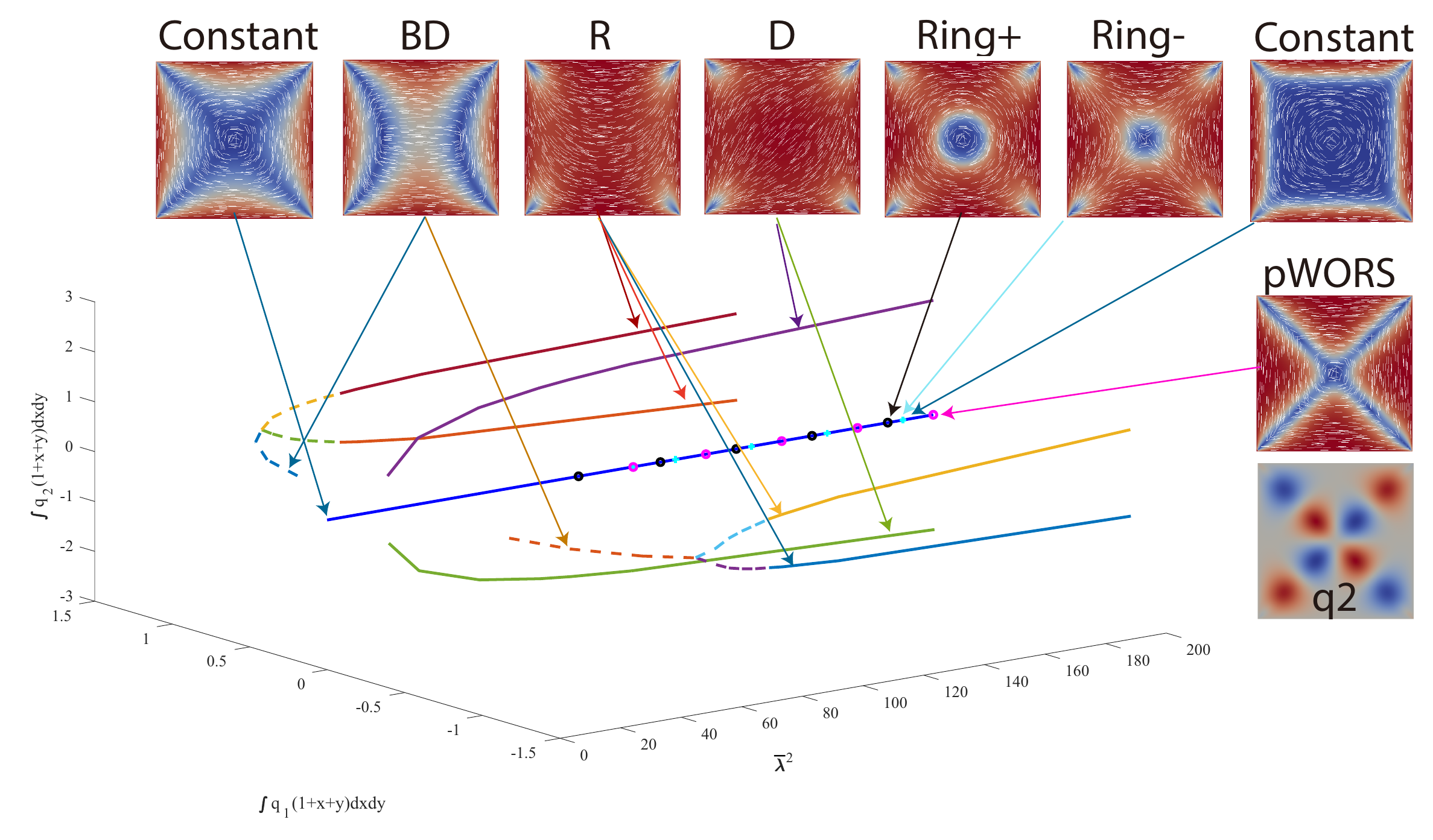}
    \end{subfigure}
    \vspace{2em}
    \begin{subfigure}{0.4\textwidth}
        \centering
        \includegraphics[width=\columnwidth]{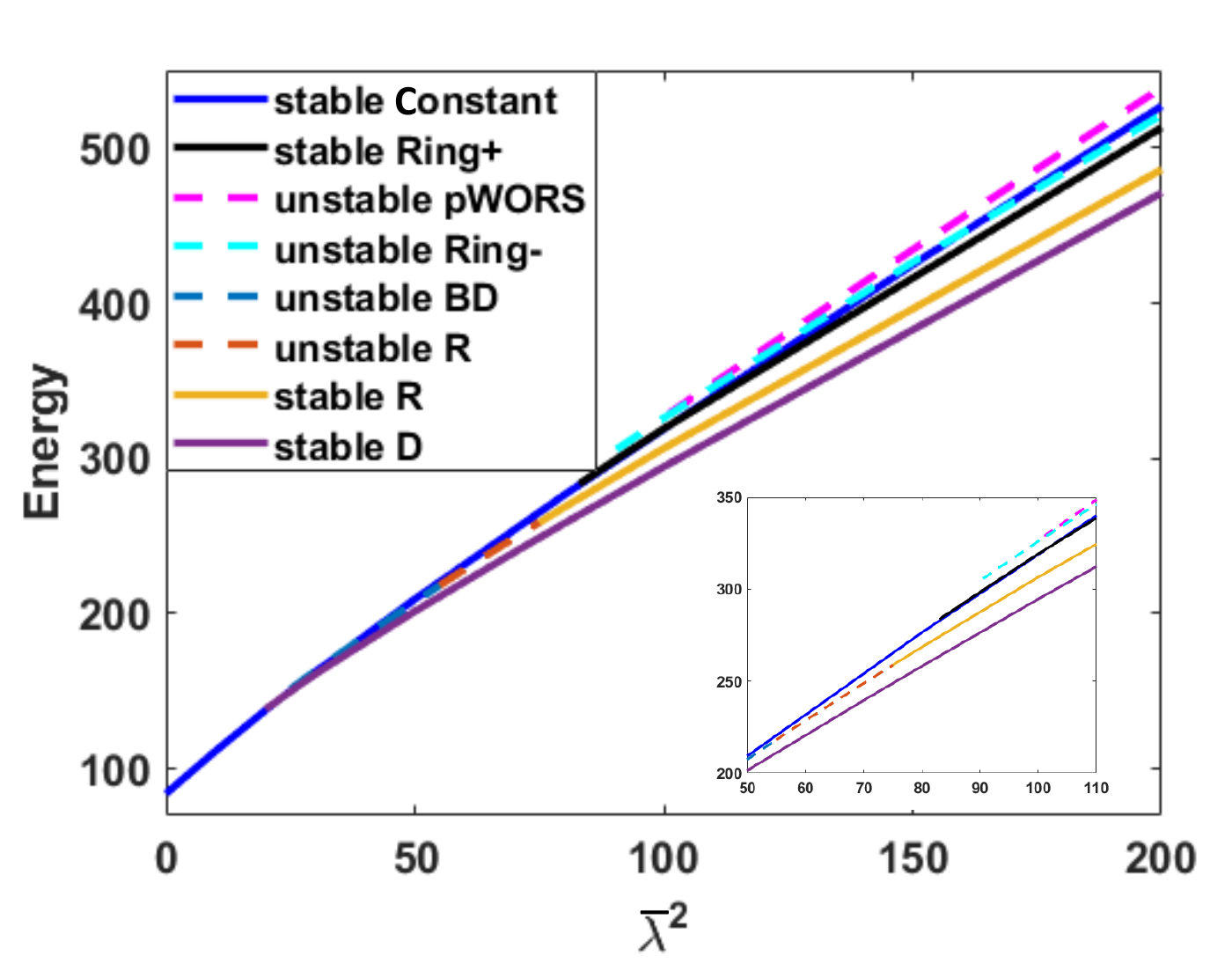}
    \end{subfigure}
    \begin{subfigure}{0.55\textwidth}
        \centering
        \includegraphics[width=0.9\columnwidth]{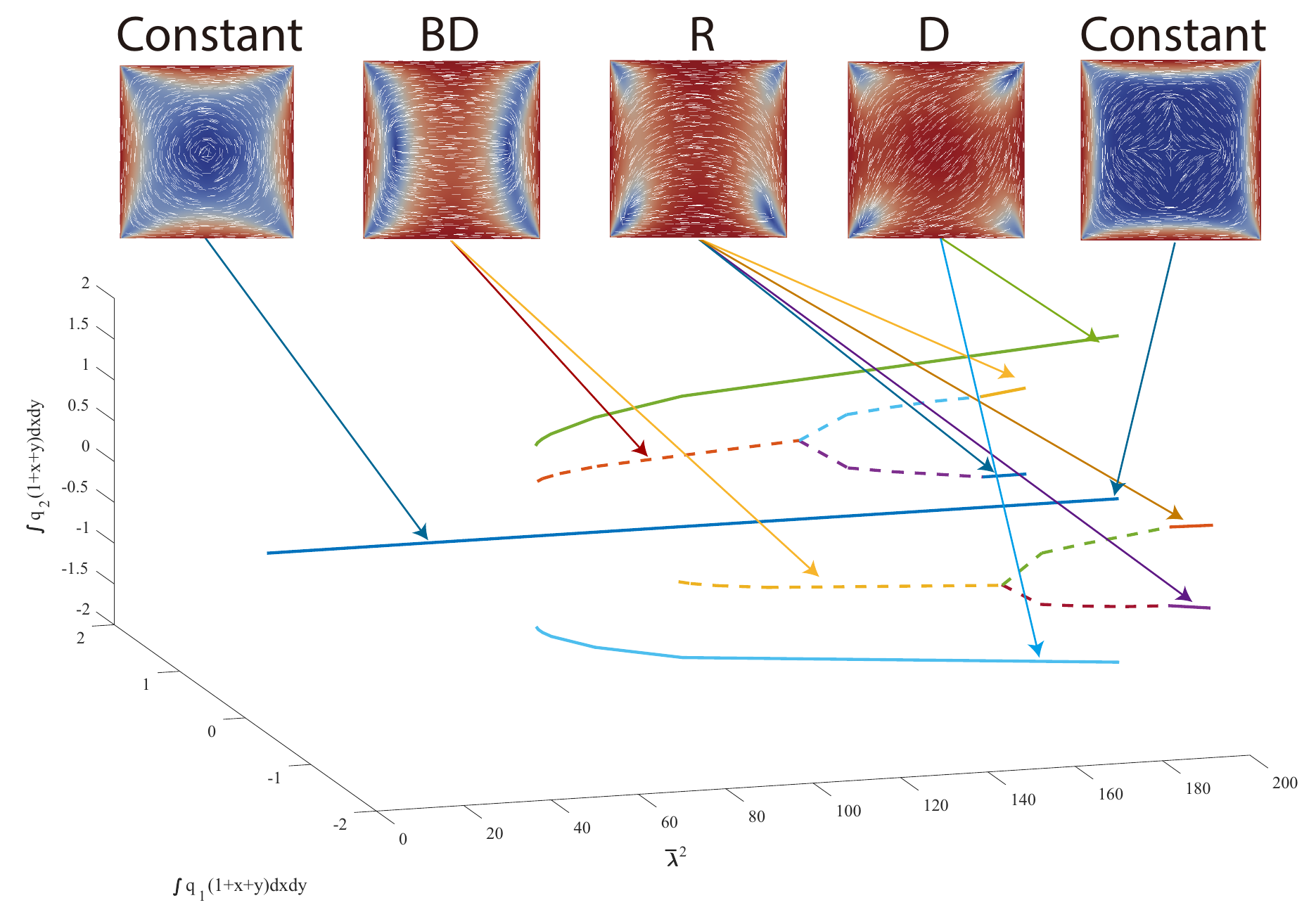}
    \end{subfigure}
    \begin{subfigure}{0.4\textwidth}
        \centering
        \includegraphics[width=\columnwidth]{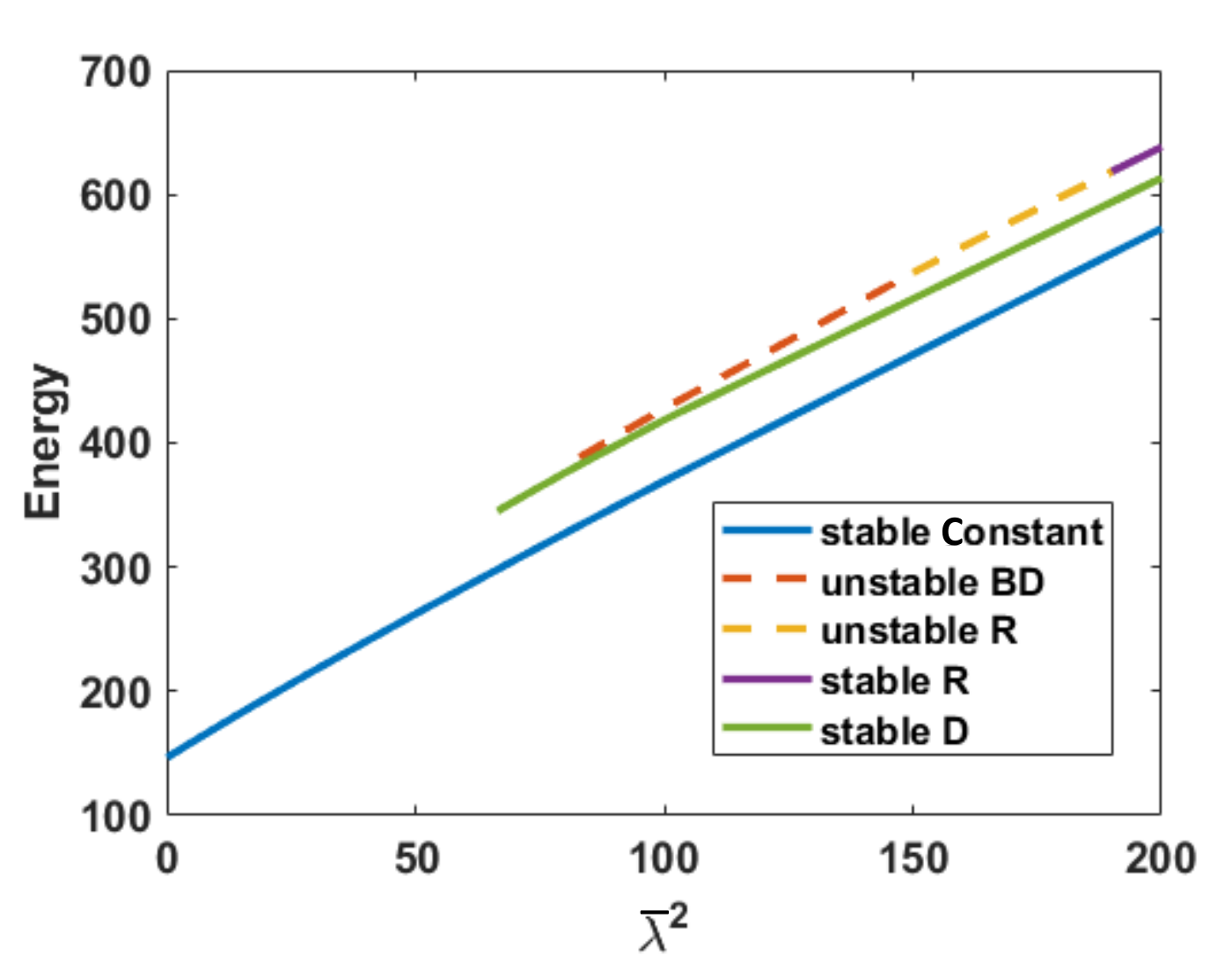}
    \end{subfigure}
        \caption{Bifurcation diagrams for the LdG model in square domain with $L_2 = 0$, $1$, $2.6$, $3$, and $10$ from top to bottom. Left: plot of $\int q_{1}\left(1+x+y\right)\textrm{d}x\textrm{d}y$, $\int q_{2}\left(1+x+y\right)\textrm{d}x\textrm{d}y$ verses $\bar{\lambda}^2$. Right: plot of the energy $J-\int_{\Omega}\min f_b\,\mathrm{dA}$ verses $\bar{\lambda}^2$.}
        \label{fig:bifurcation_diagram_10}
\end{figure}

 Consider the case $L_2 = 0$. For $\lambda < \lambda^*$, there is the unique $WORS$. For $\lambda = \lambda^*$, the stable $WORS$ bifurcates into an unstable $WORS$, and two stable $D$ solutions. When $\lambda = \lambda^{**}>\lambda^*$, the unstable $WORS$ bifurcates into two unstable $BD$, which are featured by isotropic lines or defect lines localised near a pair of opposite square edges. When $\lambda = \lambda^{***}>\lambda^{**}$, unstable $Ring^\pm$ solutions appear simultaneously. When $L_2=0$, the $Ring^+$ and $Ring^-$ solution have the same energy. Each unstable $BD$ further bifurcates into two unstable $R$ solutions. As $\lambda$ increases, the unstable $R$ solutions gain stability. The $WORS$ has the highest energy amongst the numerically computed solutions for $L_2=0$, for large $\lambda$.
For $L_2 = 1$,  the $WORS$ ceases to exist and the unique solution is the stable $Ring^+$ solution. At the first bifurcation point $\lambda = \lambda^*$, the $Ring^+$ solution bifurcates into an unstable $Ring^+$ and two stable $D$ solutions. At the second bifurcation point, $\lambda = \lambda^{**}>\lambda^*$, the unstable $Ring^+$ bifurcates into two unstable $BD$ solutions and for $\lambda = \lambda^{***}>\lambda^{**}$, the unstable $Ring^-$ and unstable $pWORS$ solution branches appear. The $Ring^-$ and $pWORS$ are always unstable and the $Ring^+$ solution has slightly lower energy than the $Ring^-$. The unstable $pWORS$ has higher energy than the unstable $Ring^\pm$ solutions when $\lambda$ is large. 
The solution landscape for $L_2=1$ and $L_2=2.6$ remain unchanged qualitatively however, for $L_2=2.6$, the unique $Ring^+$ solution for small $\lambda$ remains stable for $\bar{\lambda}^2\leq 200$ and the unstable $pWORS$ and $Ring^-$ appear for large $\lambda$. For $L_2 = 3$, the unique stable solution for small $\lambda$ is the $Constant$ solution, which is stable for $\bar{\lambda}^2\leq 200$. We can clearly see that the $Constant$ solution approaches $(q_1,q_2, q_3) \to (0,0, s_+/3)$ as $\lambda$ gets large. The $BD$ and $D$ solution branches, which were previously connected to the unique small $\lambda$ solution branch, are now disconnected from the stable $Constant$ solution branch. For $\lambda = \lambda^*$, the stable $Ring^+$ appears and for $\lambda = \lambda^{**}>\lambda^*$, the unstable $Ring^-$ and $pWORS$ appear. 
For $L_2 = 10$, the $pWORS$ and $Ring^\pm$ states disappear, and the $Constant$ solution does not bifurcate to any known states. The $BD$ and $D$ branches are disconnected from the stable $Constant$ branch. As we perform a decreasing $\bar{\lambda}^2$ sweep for the $D$ or $BD$ solution branches, we cannot find a $D$ or $BD$ solution for $\lambda<\lambda^{D}$ or $\lambda<\lambda^{BD}$, for small $\lambda^{D}$ and $\lambda^{BD}$. The $Constant$ solution has lower energy than the $R$ and $D$ solutions for large $\lambda$, as suggested by the estimates in Section~\ref{sec:lambdainfty}. 
For much larger values of $L_2$, we only numerically observe the $Constant$ solution branch.

To summarise, the primary effect of the anisotropy parameter, $L_2$, is on the unique stable solution for small $\lambda$. The elastic anisotropy destroys the cross structure of the $WORS$, and also enhances the stability of the $Ring^+$ and $Constant$ solutions. A further interesting feature for large $L_2$, is the disconnectedness of the $D$ and $R$ solution branches from the parent $Constant$ solution branch. This indicates novel hidden solutions for $L_2$, which may have different structural profiles to the discussed solution branches, and this will be investigated in greater detail, in future work.

In the next proposition, we prove a stability result which gives partial insight into the stabilising effects of positive $L_2$. Let $(q_1,q_2,q_3)$ be an arbitrary critical point of the energy functional (\ref{funcq123}).  As is standard in the calculus of variations, we say that a critical point is locally stable if the associated second variation of the energy (\ref{funcq123}) is positive for all admissible perturbations, and is unstable if there exists an admissible perturbation for which the second variation is negative. To this end, we consider perturbations of the form $\Qvec+\epsilon\mathbf{V}$, where $\mathbf{V}$ vanishes at the boundary, $\pp\Omega$. In the following proposition, we prove the stability of these critical points with respect to two classes of admissible perturbations $\mathbf{V}$, for large $L_2$. 

\begin{proposition}
For $L_2\geq\frac{\lambda^2}{L}c(A,B,C,\Omega)$, where $c$ is some constant depending only on $A,B,C$ and $\Omega$, the critical points of the energy functional (\ref{funcq123}) in the restricted admissible space
\begin{gather}
\mathcal{A}_*=\{(q_1,q_2,q_3)\in\mathcal{A}_0 :\int_\Omega|\nabla q_1|^2\leq M_1(A,B,C),\int_\Omega|\nabla q_2|^2\leq M_2(A,B,C),\int_\Omega|\nabla q_3|^2\leq M_3(A,B,C)\} ,  
\end{gather}
are locally stable with respect to the perturbations
\begin{gather}
\mathbf{V}(x,y)=v_1(x,y)(\xhat\otimes\xhat-\yhat\otimes\yhat)+v_2(x,y)(\xhat\otimes\yhat+\yhat\otimes\xhat), \label{V12}
\end{gather}
and
\begin{gather}
\mathbf{V}(x,y)=v_3(x,y)(2\zhat\otimes\zhat-\xhat\otimes\xhat-\yhat\otimes\yhat). \label{V3}
\end{gather}
\end{proposition}

\begin{proof}
To begin, consider the admissible perturbation (\ref{V3}). The second variation of the LdG energy (\ref{funcq123}) with respect to this perturbation is given by
\begin{gather}
\delta^2\mathcal{F}[v_3]=\int_\Omega(6+L_2)|\nabla v_3|^2 +\frac{\lambda^2v_3^2}{L}\left\{6A-12Bq_3+72Cq_3^2+6C(2q_1^2+2q_2^2+6q_3^2)\right\}\mathrm{dA},
\end{gather}
where $v_3\in W_0^{1,2}(\Omega)$. By an application of the Poincar\'{e} inequality and use of the relevant embedding theorem as in Proposition \ref{prop2}, there exists some constant $c_0$, which depends on the domain $\Omega$, such that
\begin{align}
\int_{\Omega}(6+L_2)|\nabla v_3|^2\,\mathrm{dA}
&\geq L_2c_0(\Omega)||v_3||_{L^4(\Omega)}^2
\end{align}
We will now restrict ourselves to studying critical points in the admissible space, $\mathcal{A}_*$, which respect the Dirichlet energy bounds for the scalar order parameters $q_1,q_2,q_3$. By applications of the H\"{o}lder inequality and further applications of the embedding theorem and Poincar\'{e} inequality in $\mathcal{A}_*$, we have that there exists some constant $c_1$ depending only on $A,B,C$ and $\Omega$ such that
\begin{align}
\delta^2\mathcal{F}[v_3]&\geq L_2c_0||v_3||_{L^4(\Omega)}^2-\frac{\lambda^2}{L}c_1(A,B,C,\Omega)||v_3||_{L^4(\Omega)}^2\\
&=\left(L_2c_0-\frac{\lambda^2}{L}c_1\right)||v_3||_{L^4(\Omega)}^2 \label{v3inequality}
\end{align}
The quantity (\ref{v3inequality}) is positive if, and only if, $L_2\geq\lambda^2c/L$ where $c:=c_1/c_0$. Similarly, we may consider the admissible perturbation (\ref{V12}). The second variation of the energy (\ref{funcq123}) with respect to this perturbation is given by
\begin{gather}
\delta^2\mathcal{F}[v_1,v_2]=\int_\Omega (2+L_2)|\nabla v_1|^2+(2+L_2)|\nabla v_2|^2+2L_2(v_{1,x}v_{2,y}-v_{1,y}v_{2,x})\,\mathrm{dA} \nonumber\\
+\frac{\lambda^2}{L}\int_\Omega (2A+4Bq_3)(v_1^2+v_2^2)+2C(2q_1^2+2q_2^2+6q_3^2)(v_1^2+v_2^2)+8C(q_1v_1+q_2v_2)^2\,\mathrm{dA}
\end{gather}
Since $v_1,v_2\in W_0^{1,2}(\Omega)$, the term
\begin{gather}
v_{1,x}v_{2,y}-v_{1,y}v_{2,x}
\end{gather}
is a null lagrangian and hence, applying the same reasoning as before, we have that there exist constants $\xi_0,\xi$, depending only on $A,B,C$ and $\Omega$ such that 
\begin{gather}
    \delta^2\mathcal{F}[v_1,v_2]\geq\left(L_2\xi_0-\frac{\lambda^2}{L}\xi_1\right)(||v_1||_{L^4(\Omega)}^2+||v_2||_{L^4(\Omega)}^2). \label{v1v2inequality}
\end{gather}
The right hand side of (\ref{v1v2inequality}) is positive if, and only if, $L_2\geq\lambda^2c/L$ where $c:=\xi_1/\xi_0$, thus completing the proof.
\end{proof}

\section{Conclusions and discussions}
\label{sec:conclusions}
We study the effects of elastic anisotropy on stable nematic equilibria on a square domain, with tangent boundary conditions, primarily focusing on the interplay between the square edge length $\lambda$ and the elastic anisotropy $L_2$. We study LdG critical points with three degrees of freedom: $q_1$ and $q_2$ which measure the degree of nematic order in the plane of the square, and $q_3$ which measures the degree of out-of-plane order in terms of the eigenvalue about the $z$-axis. We use symmetry arguments on an $1/8$-th of the square domain, to construct a LdG critical point for which $q_1$ vanishes on the square diagonals, and $q_2$ vanishes on the coordinate axes. The $WORS$ is a special class of these critical points for $L_2=0,$ with $q_2 \equiv 0$ on the square domain. In particular, $q_2$ cannot be identically zero for this LdG critical point, for $L_2 \neq 0$. This symmetric critical point is the unique LdG energy minimizer for $\lambda$ small enough, as follows from a uniqueness proof. There are different classes of these symmetric critical points for large $\lambda$. We perform asymptotic studies in the small $\lambda$ and small $L_2$ limit, and large $L_2$ limits, and provide good asymptotic approximations for the novel $Ring^+$ and $Constant$ solutions, both of which are stable for small $\lambda$ and moderate $L_2$, and large $\lambda$ and relatively large values of $L_2$, when these solutions exist. We also provide asymptotic expansions for the novel unstable $pWORS$ solution branches, featured by alternating zeroes of $q_2$ on the square diagonals. The $WORS$, $Ring^{\pm}$, $Constant$ and $pWORS$ belong to the class of symmetric critical points constructed in Proposition~\ref{prop:forever_critical}.
The large $\lambda$-picture for $L_2 \neq 0$ is qualitatively similar to the $L_2 = 0$ case, with the stable diagonal, $D$ and rotated, $R$ solutions. The notable difference is the emergence of the competing stable $Constant$ solution for large $L_2$, which is energetically preferable to the $D$ and $R$-solutions, for large $L_2$ and large $\lambda$. This suggests that for highly anisotropic materials with large $L_2$, the experimentally observable state is the $Constant$ solution with $q_1^2 + q_2^2 \approx 0$ in the square interior. In other words, the $Constant$ state is almost perfectly uniaxial, with uniaxial symmetry along the $z$-direction, and will offer highly contrasting optical properties compared to the $D$ and $R$ solutions. This offers novel prospects for multistability for highly anisotropic materials.

Another noteworthy feature is the stabilising effect of $L_2$, as discussed in Section~\ref{sec:bifurcations}. The $Ring^+$ solution has a central point defect in the square interior and is unstable for $L_2= 0$. However, it gains stability for moderate values of $\lambda$, as $L_2$ increases, and ceases to exist for very large positive values of $L_2$. We note some similarity with recent work on ferronematics \cite{hanwaltonharrismajumdar2021}, where the coupling between the nematic director and an induced spontaneous magnetisation stabilises interior nematic point defects, with $L_2= 0$. It remains an open question as to whether elastic anisotropy or coupling energies (perhaps with certain symmetry and invariance properties) can stabilise interior nematic defects, and help us tune the locations, dimensionality and multiplicity of defects for tailor-made applications.

\section*{Acknowledgments}
AM ackowledges support from the University of Strathclyde New Professors Fund and a University of Strathclyde Global Engagement Grant. AM is also supported by a Leverhulme International Academic Fellowship. AM and LZ acknowledge the support from Royal Society Newton Advanced Fellowship. LZ acknowledges support from the National Natural Science Foundation of China No. 12050002. YH acknowledges support from a Royal Society Newton International Fellowship.

\bibliographystyle{unsrt}
\bibliography{bibliography1-2}

\end{document}